\documentclass[a4paper,11pt,reqno]{amsart}
\usepackage[foot]{amsaddr}
\usepackage{amssymb}
\usepackage{epsfig}
\usepackage{amsfonts}
\usepackage{amsmath}
\usepackage{euscript}
\usepackage{amscd}
\usepackage{amsthm}
\usepackage{enumitem}
\DeclareMathAlphabet{\mathpzc}{OT1}{pzc}{m}{it}
\usepackage{color}
\usepackage{mathtools}
\usepackage{xcolor}
\usepackage{soul}
\usepackage{graphicx}
\usepackage{wrapfig}
\usepackage{subcaption}
\usepackage[nopar]{lipsum}
\usepackage[pagebackref, hypertexnames=false, colorlinks, citecolor=red, linkcolor=blue, urlcolor=red]{hyperref}
\usepackage[marginratio=1:1,height=584pt,width=380pt,tmargin=112pt]{geometry}
\usepackage{pgf,tikz,pgfplots}
\usetikzlibrary{shapes,arrows}
\usepackage{adjustbox}
\usepackage{setspace}
\usepackage[numbers,sort&compress]{natbib}

\usetikzlibrary{calc}
\usetikzlibrary{patterns}
\pgfplotsset{compat=1.14}

\usepackage{marginnote}

\usepackage[normalem]{ulem}
\newcommand{\updownextend}[1]{
\addtolength{\textheight}{#1}}
\updownextend{2cm}
\allowdisplaybreaks[4]

\usepackage{mathrsfs}
\DeclareFontFamily{OT1}{pzc}{}
\DeclareFontShape{OT1}{pzc}{m}{it}{<-> s * [1.10] pzcmi7t}{}
\DeclareMathAlphabet{\mathpzc}{OT1}{pzc}{m}{it}

\DeclareSymbolFont{SY}{U}{psy}{m}{n}
\DeclareMathSymbol{\emptyset}{\mathord}{SY}{'306}

\theoremstyle{plain}

\newtheorem{thm}{Theorem}[section]
\newtheorem*{thm*}{Theorem}
\newtheorem{cor}[thm]{Corollary}
\newtheorem{lem}[thm]{Lemma}
\newtheorem{prop}[thm]{Proposition}
\newtheorem{defn}[thm]{Definition}
\newtheorem{rem}[thm]{Remark}

\newtheoremstyle{named}{}{}{\itshape}{}{\bfseries}{.}{.5em}{#1 \thmnote{#3}}
\theoremstyle{named}

\numberwithin{equation}{section}

\def\C{{\mathbb C}}

\def\norm#1{\left\|{#1}\right\|}

\def\N{\mathbb{N}}

\def\l{\lambda}

\def\lo{\longrightarrow}

\def\m{\mathcal}
\def\mb{\mathbb}

\def\a{\alpha}
\def\b{\beta}
\def\g{\gamma}

\def\beq{\begin{eqnarray}}
\def\eeq{\end{eqnarray}}
\def\beqa{\begin{eqnarray*}}
\def\eeqa{\end{eqnarray*}}

\def\bl{\boldsymbol}

\def\i{\prime}

\def\bl{\boldsymbol}
\def\mf{\mathbf}
\def\ca{{\rm card}}
\newcommand{\overbar}[1]{\mkern 1.5mu\overline{\mkern-1.5mu#1\mkern-1.5mu}\mkern 1.5mu}

\newcommand{\be}{\begin{equation}}
\newcommand{\ee}{\end{equation}}
\newcommand{\bea}{\begin{eqnarray}}
\newcommand{\eea}{\end{eqnarray}}
\newcommand{\Bea}{\begin{eqnarray*}}
\newcommand{\Eea}{\end{eqnarray*}}
\newcommand{\inner}[2]{\langle #1,#2 \rangle }%

\newcounter{cnt1}
\newcounter{cnt2}
\newcounter{cnt3}
\newcommand{\blr}{\begin{list}{$($\roman{cnt1}$)$}
 {\usecounter{cnt1} \setlength{\topsep}{0pt}
 \setlength{\itemsep}{0pt}}}
\newcommand{\bla}{\begin{list}{$($\alph{cnt2}$)$}
 {\usecounter{cnt2} \setlength{\topsep}{0pt}
 \setlength{\itemsep}{0pt}}}
\newcommand{\bln}{\begin{list}{$($\arabic{cnt3}$)$}
 {\usecounter{cnt3} \setlength{\topsep}{0pt}
 \setlength{\itemsep}{0pt}}}
\newcommand{\el}{\end{list}}

\title[analytic structure of weighted shifts]{On analytic structure of weighted shifts on generalized directed semi-trees}
\author[Ghosh, Hazra]{Gargi Ghosh and Somnath Hazra}
\address[Ghosh]{Department of Mathematics, Indian Institute of Science, Bangalore, 560012, India}
\address[Hazra]{Department of Mathematics \& Statistics, Indian Institute of Science Education and Research Kolkata, Mohanpur, 741246, India}
\email[Ghosh]{gargighosh@iisc.ac.in}
\email[Hazra]{somnath.hazra.2008@gmail.com}
\thanks{The research of the first named author was supported by IoE-IISc Fellowship. The research of the second named author was supported by a post-doctoral fellowship of the NBHM}
\begin{document}
\begin{abstract}
 Inspired by  natural classes of examples, we define generalized directed semi-tree and construct weighted shifts on the generalized directed semi-trees. Given an $n$-tuple of directed directed semi-trees with certain properties, we associate an $n$-tuple of multiplication operators on a Hilbert space $\mathscr{H}^2(\beta)$ of formal power series.  
 Under certain conditions, $\mathscr{H}^2(\beta)$ turns out to be a reproducing kernel Hilbert space consisting of holomorphic functions on some domain in $\mb C^n$ and the $n$-tuple of multiplication operators on $\mathscr{H}^2(\beta)$ is unitarily equivalent to an $n$-tuple of weighted shifts on the generalized directed semi-trees. Finally, we exhibit two classes of examples of $n$-tuple of operators which can be intrinsically  identified as weighted shifts on generalized directed semi-trees.
\end{abstract}
\maketitle
\section{Introduction}
The study of the adjacency operator for an infinite directed graph has been  initiated in \cite{MR1022150}. Later Jablonski, Stochel and Jung specialize on adjacency operators for weighted directed trees in \cite{MR2919910}. They refer it as the weighted shifts on directed trees and yield various significant results that enrich operator theory in several aspects. Inspired by generalized creation operators on Segal-Bargmann space, Majdak and Stochel have generalized it to weighted shifts on the directed semi-trees in \cite{MR3552930}.


Let $\mb D$ denote the open unit disc in the complex plane $\mb C.$  For $\l >0,$ it is well known that
\Bea
K^{(\lambda)} (\bl z,\bl w) := \displaystyle \prod_{i = 1}^n \frac{1}{(1 - z_i\overbar{w_i})^{\lambda}},\,\,\bl z, \bl w \in \mathbb D^n
\Eea
is a positive definite kernel on $\mb D^n.$ Let $\mathbb{A}^{(\lambda)}(\mathbb D^n)$ denote the Hilbert space consisting of holomorphic functions with reproducing kernel $K^{(\lambda)}.$
In particular, $\mathbb{A}^{(\lambda)}(\mathbb D^n)$ coincides with Hardy space $H^2(\mb D^n)$ and the Bergman space $\mb A^2(\mb D^n)$  for $\l =1$ and $\l=2, $ respectively. Let us denote the permutation group on $n$ symbols by $\mathfrak S_n.$ The subspaces \Bea
\mb A_{\rm sym}^{(\l)}(\mb D^n) =\{f\in \mb A^{(\l)}(\mb D^n):f\circ\sigma^{-1}=f \mbox{~for~} \sigma\in\mathfrak S_n\},
\Eea and
\Bea
\mb A_{\rm anti}^{(\l)}(\mb D^n)=\{f\in \mb A^{(\l)}(\mb D^n):f\circ\sigma^{-1}={\rm sgn} (\sigma)f \mbox{~for~} \sigma\in\mathfrak S_n\}
\Eea are joint reducing subspaces of the $n$-tuple of multiplication operators $M_{\bl s}:= (M_{s_1},$ $\ldots,$ $M_{s_n})$ on $\mb A^{(\lambda)}(\mb D^n),$ where $s_k$ denotes the elementary symmetric polynomial of degree $k$ in $n$ variables, see \cite[p. 774]{MR3906291}.
In \cite{BDGS}, Biswas et al. prove that each of the operators $M_{\bl s}|_{\mb A_{\rm anti}^{(\l)}(\mb D^n)}$ and $M_{\bl s}|_{\mb A_{\rm sym}^{(\l)}(\mb D^n)}$ is unitarily equivalent to the $n$-tuple of coordinate multiplication operators on some reproducing kernel Hilbert space containing holomorphic functions on symmetrized polydisc.
However, in \cite[p. 771, Corollary 3.9]{MR3906291} it is shown that none of $M_{\bl s}|_{\mb A_{\rm anti}^{(\l)}(\mb D^n)}$ and $M_{\bl s}|_{\mb A_{\rm sym}^{(\l)}(\mb D^n)}$ is unitarily equivalent to any joint
weighted shift. It has been observed that each of the operators $M_{\bl s}|_{\mb A_{\rm anti}^{(\l)}(\mb D^n)}$ and $M_{\bl s}|_{\mb A_{\rm sym}^{(\l)}(\mb D^n)}$ has a natural identification as an $n$-tuple of weighted shift operators on generalized structure of directed semi-trees. Motivated from these examples, we have defined generalized directed semi-tree and described weighted shifts on generalized directed semi-trees. 

We have represented weighted shifts on generalized directed semi-trees as multiplication operators on Hilbert spaces consisting of analytic functions. Representing weighted shift operators as multiplication operators makes various well known  operator theoretic tools accessible to analyze these operators. For example, comprehending the unilateral shift to coordinate multiplication operator on the Hardy space of unit disc yields a significant exposition. Shields exhibits an insightful association of weighted shift operators with analytic functions in \cite{MR0361899}. Jewell and Lubin show a similar interplay between commuting weighted shifts and analytic functions in several variables, see \cite{MR532875}.  Recently, in \cite{MR3532172}, Chavan et al. described an analytic model for left-invertible weighted shifts on directed trees using Shimorin's analytic model, described in \cite{MR1810120}. A different approach has been made to describe an analytic structure of weighted shifts on directed trees in \cite{MR3683793}. However, we follow the framework of \cite{MR532875} in this paper. 

Now we briefly outline the content of this paper. In Section \ref{section 2}, we have reproduced a few basic notions of graph theory. 
Given generalized directed semi-trees $(\m G_i,m_i)$, $1 \leq i \leq n$, an $n$-tuple of multiplication operators on a Hilbert space of formal power series $\mathscr{H}^2(\beta)$ have been constructed in Section \ref{ana}. A necessary and sufficient condition for the continuity of that $n$-tuple of the multiplication operators on $\mathscr{H}^2(\beta)$ is provided in Lemma \ref{con}. Under certain conditions, $\mathscr{H}^2(\beta)$ turns out to be a reproducing kernel Hilbert space consisting of holomorphic functions on some domain in $\mb C^n.$ Moreover, we show in Theorem \ref{lem} that the $n$-tuple of multiplication operators on $\mathscr{H}^2(\beta)$ is unitarily equivalent to an $n$-tuple of operators $(\Lambda_1,\ldots,\Lambda_n),$ where each $\Lambda_i$ is a weighted shift on the generalized directed semi-tree $(\m G_i,m_i)$.

In Section \ref{example} and Section \ref{section 5}, we provide the following natural classes of examples.
\begin{itemize}
    \item For $\l>0$, we denote the weighted Bergman space on polydisc $\mb D^n$ by $\mb A^{(\l)}(\mb D^n)$. Moreover, the subspaces $\mb A_{\rm sym}^{(\l)}(\mb D^n)$ and $\mb A_{\rm anti}^{(\l)}(\mb D^n)$ consist all symmetric functions in $\mb A^{(\l)}(\mb D^n)$ and all anti-symmetric functions in $\mb A^{(\l)}(\mb D^n)$, respectively. The elementary symmetric polynomial of degree $i$ in $n$ variables is denoted by $s_i.$ For each $1 \leq i \leq n$, the restriction of the multiplication operator $M_{s_i}$ on the spaces $\mb A_{\rm sym}^{(\l)}(\mb D^n)$ and $\mb A_{\rm anti}^{(\l)}(\mb D^n)$ have a natural identification as weighted shifts on generalized directed semi-trees. 
    \item Let $G$ be a finite pseudoreflection group and $D$ be a complete Reinhardt domain in $\mb C^n$ which is $G$-invariant. The Bergman space on $D$ is denoted by $\mb A^2(D).$ 
    For each $1 \leq i \leq n$, the operator $M_{\theta_i} : \mathbb A^2(D) \to \mathbb A^2(D)$ is unitarily equivalent to a weighted shift on a generalized directed semi-tree, where $\{\theta_i:i=1,\ldots,n\}$ is a set of basic polynomials associated to the group $G$.
\end{itemize}

\section{generalized directed semi-trees}\label{section 2}
\subsection{Basic notions of Graph theory}
We begin by recalling a number of useful definitions from graph theory.
For a nonempty set $V$ and a subset $E \subseteq (V \times V)\backslash \{(v,v) : v \in V\}$ of ordered tuples, we denote the \textit{directed graph} $\mathcal{G}$ by the pair $\mathcal{G} = (V,E)$. 
An element of $V$ is called a \textit{vertex} and an element of $E$ is called an \textit{edge}. We enlist some requisite definitions related to a directed graph below, following the notations described in \cite[p. 1429]{MR3552930}.
\begin{enumerate}
    \item[1.] A directed graph $\mathcal{G}=(V,E)$ is said to be \textit{connected} if for every two distinct vertices $u$ and $v$, there exists a finite sequence $v_1, v_2,...., v_n \in V,$ for $ n \geq 2,$ such that $u= v_1,$ either $(v_j,v_{j+1}) \in E$ or $(v_{j+1},v_j) \in E$  for all $j= 1,2,....,n-1$ and $v_n= v.$
    \item[2.] A finite sequence $v_1, v_2,...., v_n $ $(n \geq 2)$ of distinct vertices of $\mathcal{G} $ is said to be a \textit{circuit} if $(v_j,v_{j+1}) \in E$ for all $j = 1,2,....,n-1$ and $(v_n,v_1) \in E .$
    \item[3.] For any $u \in V,$ the children of $u$ and the parents of $u$ are given by
\begin{enumerate}
\item $\textnormal{Chi}(u):=\{v\in V : (u,v) \in E\}$ and
\item $\textnormal{Par}(u):=\{v\in V : (v,u) \in E\}$, respectively.
\end{enumerate}
 \item[4.] A vertex $v$ is called a \textit{root} of the graph $\m G$ if $\textnormal{Par}(v)$ is empty. The set of all roots of $\mathcal{G}$ is denoted by Root($\mathcal{G}$). we also fix $V^{\circ} = V \backslash$  Root($\mathcal{G}$).
 \item[5.] For a vertex $w \in V$, we fix $\textnormal{Chi}^{(0)}(w)= \{w\}$ and $\textnormal{Par}^{(0)}(w)= \{w\}$. For $n \in \mb N,$ $n$-th children of $w$ and $n$-th parents of $w$ is denoted by
\begin{enumerate}
    \item $\textnormal{Chi}^{(n)}(w)= \textnormal{Chi}(\textnormal{Chi}^{(n-1)}(w)),$ and
    \item $\textnormal{Par}^{(n)}(w)= \textnormal{Par}(\textnormal{Par}^{(n-1)}(w))$, respectively.
\item In addition to it, \textit{descendants} of $w$ is defined as
$$Des(w) = \cup_{n = 0}^{\infty} \textnormal{Chi}^{(n)}(w).$$
\end{enumerate}  
\end{enumerate}
Now we are in a position to recall the definition of a directed tree from \cite[p. 10]{MR2919910}.
\begin{defn}
A directed graph $\mathcal{G} = (V,E)$  is called a directed tree if the following conditions are satisfied.
\begin{enumerate}
\item[$\rm (i)$]
$\mathcal{G}$ has no circuit.
\item[$\rm (ii)$]
$\mathcal{G}$ is connected.
\item[$\rm (ii)$]
For each vertex $v \in V^{\circ},$ the set $\textnormal{Par}(v)$ has only one element.
\end{enumerate}
\end{defn}

Recently, Majdak and Stochel generalized the notion of a directed tree and defined a directed semi-tree in \cite[p. 1430]{MR3552930} as following: 
\begin{defn}\label{def:0.2}
A directed graph $\mathcal{G} = (V,E)$  is called a directed semi-tree if the following conditions are satisfied.
\begin{enumerate}
\item[$\rm (i)$]
$\mathcal{G}$ has no circuit.
\item[$\rm (ii)$]
$\mathcal{G}$ is connected.
\item[$\rm (iii)$]
$\ca(\textnormal{Chi}(u)\cap \textnormal{Chi}(v)) \leq 1$ for distinct all $u,v \in V,$ where $\ca(\textnormal{Chi}(u)\cap \textnormal{Chi}(v))$ represents the cardinality of the set $\textnormal{Chi}(u)\cap \textnormal{Chi}(v)$. 
\item[$\rm (iv)$]
For all $u,v \in V ;$ there exists $w \in V $ such that $u,v \in Des(w).$
\end{enumerate}
\end{defn}
The motivation to generalize the notion of directed semi-tree comes from a number of examples which occur naturally. A few such examples are discussed in Section \ref{example} and Section \ref{section 5} which emphasize that this generalization is not superficial. 
\begin{defn}\label{gensem}
A directed graph  $\m G = (V,E)$ is called generalized directed semi-tree if the following conditions are satisfied.
\begin{enumerate}
\item[$\rm (i)$]
$\m G$ has no circuit.
\item[$\rm (ii)$]
$\m G$ has countable connected components.
\item[$\rm (iii)$]
There exists a fixed non-negative integer, say $m$, such that 
for every distinct $u,v \in V;$ \bea\label{m}\ca(\textnormal{Chi}(u) \cap \textnormal{Chi}(v)) \leq m.\eea
\end{enumerate}
\end{defn}
\includegraphics[scale=0.50]{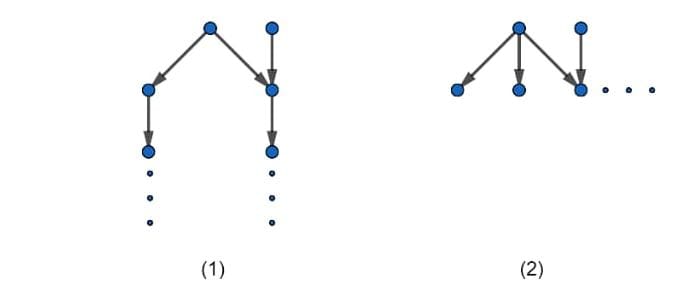}

The directed graphs (1) and (2) are examples of generalized directed semi-tree. Since each of the graphs (1) and (2) has two roots, they are not directed semi-trees.

Note that the cardinality of $\textnormal{Root}(\m G)$ can be countably infinite. If any non-negative integer $m$ satisfies Equation \eqref{m}, then for every natural number $M>m,$ we have $\ca(\textnormal{Chi}(u) \cap \textnormal{Chi}(v)) \leq M.$ So we fix the customary to choose the least non-negative integer that satisfies Equation \eqref{m}. We denote the generalized directed semi-tree by $(\m G, m)$ in the sequel, where $m$ is the least non-negative integer satisfying Equation \eqref{m}. The next proposition manifests a relation between the class of generalized directed semi-trees and the class of directed trees.
\begin{prop}
A generalized directed semi-tree $(\m G,0)$ is a directed tree if and only if it is connected.
\end{prop}
\begin{proof}
Forward direction of the proposition follows immediately. Now we prove the backward direction. Assume that $\m G$ is connected. Since $\m G$ is a generalized directed semi-tree, it has no circuit.
Arguing by contradiction, suppose that there exists $v_\circ \in V^{\circ}$ such that $\textnormal{Par}(v_{\circ})$ has more than one element, that is, $\ca(\textnormal{Par}(v_{\circ})) > 1.$
Let $u_1, u_2 \in \textnormal{Par}(v_{\circ}) , u_1 \neq u_2.$ Then $\textnormal{Chi}(u_1) \cap \textnormal{Chi}(u_2) \supseteq \{v_{\circ}\}.$ This implies that $\ca(\textnormal{Chi}(u_1) \cap \textnormal{Chi}(u_2)) \geq 1,$ which contradicts our hypothesis.
Hence for all $v \in V^{\circ},$ $\textnormal{Par}(v)$ has only one element.
This shows that $\m G$ is a directed tree.
\end{proof}
\begin{rem}
Note that a generalized directed semi-tree $(\m G,1)$ is a directed semi-tree if
\begin{enumerate}
\item[$\rm (i)$]
$\ca(Root(\m G)) \leq 1,$
\item[$\rm (ii)$]
For all $u,v \in V ;$ there exists $w \in V $ such that $u,v \in Des(w).$
\end{enumerate}
\end{rem}

\section{Analytic structure for a Weighted Shift on generalized directed semi-tree}\label{ana}
\subsection{Weighted shift on generalized directed semi-tree}
Let $(\m G,m) = (V,E)$ be a generalized directed semi-tree. 
We assign a complex number $\l_{(\mf u, \mf v)}$ to each edge $(\mf u, \mf v) \in E$ such that the following holds for every $\mf v \in V:$ 
\bea\label{inside l2}
\sum_{\mf u \in \textnormal{Par}(\mf v)}| \l_{(\mf u,\mf v)}|^2 < \infty.
\eea
We refer $\l_{(\mf u, \mf v)}$ as the weight on the edge $(\mf u,\mf v).$ 
The Hilbert space $\ell^2(V)$ is the family of all square-summable complex-valued functions on $V$ with the standard inner product $\inner{f}{g} := \sum_{\mf v \in V} f(\mf v) \overbar{g(\mf v)}$ for $f, g \in \ell^2(V).$ 
 Let $T_{\m G}$ be the operator defined on the set of all complex-valued functions on $V$  by
 \bea\label{tf} (T_{\m G } f)(\mf v) =
 \begin{cases}
      \sum_{\mf u \in \textnormal{Par}(\mf v)} \l_{(\mf u,\mf v)} f(\mf u),    &  \text{    if     }  \mf v \in V^{\circ},\\
      0, & \text{     if   } \mf v \in \text{Root}(\m G).
   \end{cases}
 \eea
We denote $\mathcal{D}(\Lambda_{\m G}) := \{ f \in \ell^2(V) : T_{\m G } f  \in \ell^2(V) \}.$ The operator $\Lambda_{\m G } : \mathcal{D}(\Lambda_{\m G}) \longrightarrow  \ell^2(V),$ defined by \bea\label{ff} \Lambda_{\m G}f = T_{\m G } f,\,\,f \in \mathcal{D}(\Lambda_{\m G}),\eea is called the weighted shift operator on the generalized directed semi-tree $\m G$ with weights $\{\l_{(\mf u, \mf v)} : (\mf u, \mf v) \in E\}.$ 

For an element $\mf v \in V,$ if the set $\textnormal{Par}(\mf v)$ is finite, $(\Lambda_{\m G}f)(\mf v)$ is well-defined for every $f \in \m D(\Lambda_{\m G})$. On the other hand, if for some $\mf v \in V,$ $\textnormal{Par}(\mf v) = \{ \mf u_i : i \in \N \},$ then by Cauchy-Schwarz inequality we have
\Bea
|\sum_{i= 1}^{m}  \l_{(\mf u_i,\mf v)} f(\mf u_i) - \sum_{i= 1}^{k}  \l_{(\mf u_i,\mf v)} f(\mf u_i)  | \leq \big(\sum_{i= k+1}^{m} |\l{(\mf u_i,\mf v)}|^2\big)^{\frac{1}{2}} \big(\sum_{i= k+1}^{m} |f(\mf u_i)|^2\big)^{\frac{1}{2}},
\Eea
for every $m,k \in \mb N$ with $k < m .$
From the Equation $\eqref{inside l2}$, we have $\sum_{i= k+1}^{m} |\l{(\mf u_i,\mf v)}|^2$ $\rightarrow 0$ as $m,k \rightarrow \infty$. Moreover, since $f \in  \ell^2(V),$ it follows that $\sum_{i= k+1}^{m} |f(\mf u_i)|^2 \rightarrow 0$ as $m,k \rightarrow \infty.$ Therefore, $\sum_{\mf u \in \textnormal{Par}(\mf v)} \l_{(\mf u,\mf v)} f(\mf u)$ is convergent and hence $(\Lambda_{\m G }f)(\mf v)$ is well defined.
\begin{rem}
In case, $\m G$ is a directed tree or a directed semi-tree, the definition of  $\Lambda_{\m G }$  coincides with the definition of weighted shift on directed tree and directed semi-tree, respectively, see \cite[Definition 3.1.1]{MR2919910} and \cite[Equation (5.2), p. 1437]{MR3552930}.

\end{rem}
For each $\mf u \in V,$ let $\chi_{\mf u} : V \longrightarrow \C$ be defined by
\bea\label{eqn w1}
\chi_{\mf u}(\mf v) =
 \begin{cases}
      1, &  \text{    if     }  \mf v= \mf u,\\
      0, & \text{otherwise.}
   \end{cases}
\eea
From Equation \eqref{tf} and Equation \eqref{ff}, $(\Lambda_{\m G}\chi_{\mf u})(\mf v)=  \sum_{\mf w \in \textnormal{Par}(\mf v)} \l_{(\mf w,\mf v)} \chi_{\mf u}(\mf w).$ Therefore,
\Bea\label{def of lambda}
(\Lambda_{\m G}\chi_{\mf u})(\mf v) =
\begin{cases}
     \l_{(\mf u,\mf v)},& \text{~if~} \mf u \in {\rm Par}(\mf v),\\
     0,& \text{~otherwise,~}
\end{cases}
\Eea
which can be rewritten as \bea
\label{chi} \Lambda_{\m G}\chi_{\mf u} = \sum_{\mf v \in {\rm Chi}(\mf u)} \lambda_{(\mf u,\mf v)} \chi_{\mf v}.
\eea
 If ${\rm Chi}(\mf u)$ is finite for every $\mf u \in V,$ then $\{ \chi_{\mf u}\}_{\mf u \in V} \subseteq \m D(\Lambda_{\m G}).$ However, this does not ensure the boundedness of the operator $\Lambda_{\m G}.$ The following proposition provides a necessary and sufficient condition for the boundedness of $\Lambda_{\m G}.$

\begin{prop}
Let $\m G=(V,E)$ be a generalized directed semi-tree with the property that $\ca({\rm Par}(\mf v)) \leq k$ for every $\mf v \in V$ and $\Lambda_{\m G}$ be the weighted shift operator on $\m G$ with weights $\{\l_{(\mf u, \mf v)} : (\mf u, \mf v) \in E\}.$ The operator $\Lambda_{\m G}$ is bounded on $\ell^2(V)$ if and only if $\text{sup}_{\mf{v} \in V} \sum_{\mf{u} \in \text{\textnormal{Chi}}(\mf{v})} |\l_{(\mf v, \mf u)}|^2 < \infty .$
\end{prop}
\begin{proof}
Suppose that $\Lambda_{\m G}$ is bounded on $\ell^2(V),$ that is, the operator norm of $\Lambda_{\m G}$ is finite and suppose that $\norm{\Lambda_{\m G}} = c$. For $\mf v \in V$, it follows from Equation \eqref{chi} that the norm of $\Lambda_{\m G}\chi_{\mf v}$ in $\ell^2(V)$ is given by $\norm{\Lambda_{\m G}\chi_{\mf v}}^2 =  \sum_{\mf{u} \in \text{\textnormal{Chi}}(\mf{v})} |\l_{(\mf v, \mf u)}|^2$. Since $\norm{\Lambda_{\m G}\chi_{\mf v}}^2 \leq c^2,$ the result follows.

Conversely, suppose that $\text{sup}_{\mf{v} \in V} \sum_{\mf{u} \in \text{\textnormal{Chi}}(\mf{v})} |\l_{(\mf v, \mf u)}|^2 =c^\prime.$ Let $f \in \ell^2(V)$ be such that $f = \sum_{\mf u \in W} f_{\mf u} \chi_{\mf u},$ where $W$ is a finite subset of $V$. Then we have $$\Lambda_{\m G}f = \sum_{\mf u \in W} f_{\mf u} (\Lambda_{\m G}\chi_{\mf u}) = \sum_{\mf u \in W} f_{\mf u} \sum_{\mf v \in {\rm Chi}(\mf u)} \lambda_{(\mf u,\mf v)} \chi_{\mf v} = \sum_{\mf v \in W} \big(\sum_{\mf u \in {\rm Par}(\mf v)} \l_{(\mf u, \mf v)} f_{\mf u}\big) \chi_{\mf v}.$$ Therefore, 
\Bea 
\norm{\Lambda_{\m G}f}^2 
= \sum_{\mf v \in W} |\sum_{\mf u \in {\rm Par}(\mf v)} f_{\mf u} \l_{(\mf u, \mf v)}|^2 
&\leq& k \sum_{\mf v \in W} \sum_{\mf u \in {\rm Par}(\mf v)} |f_{\mf u} \l_{(\mf u, \mf v)} |^2 \\ 
&=& k \sum_{\mf u \in W} (\sum_{\mf v \in {\rm Chi}(\mf u)} | \l_{(\mf u, \mf v)} |^2)|f_{\mf u}|^2 \\ 
&\leq& kc^\prime\norm{f}^2. \Eea
Note that $\ca({\rm Par}(\mf v)) \leq k$ for every $\mf v \in V$, so the first inequality follows by applying Cauchy-Schwarz inequality on the constant function $1$ and $f_{\mf u} \l_{(\mf u, \mf v)}$. Thus the operator $\Lambda_{\m G}$ is bounded on a dense subset of $\ell^2(V)$. Therefore, $\Lambda_{\m G}$ is bounded on $\ell^2(V)$.
\end{proof}

\subsection{Analytic Structure}
Given $n$ generalized directed semi-trees, our aim is to determine an $n$-tuple of multiplication operators on a Hilbert space of formal power series which is a reproducing kernel Hilbert space consisting of holomorphic functions under certain condition. We restrict our attention to the generalized directed semi-trees with following properties.

Suppose $\mb N$ is the set of all natural numbers and $\mb N_0 = \mb N \cup \{0\}.$ Let $V \subseteq \mathbb{N}_0^n$ and $(\m G,l) = (V,E)$ be a generalized directed semi-tree with the following properties:
\begin{enumerate}[label=(\arabic*)]
 \item \label{10} If $(\mf{v},\mf{u}) \in E$ then $v_j \leq u_j$ for $j \in \{1,2, \ldots, n\}.$
 \item \label{20}There exists a natural number $m$ such that $\ca(\text{\textnormal{Chi}}(\mf{u})) \leq m$ for every $\mf{u} \in V$ and $V_m = \{\mf v \in V: \ca({\rm Chi}(\mf v)) = m\}$ is non-empty. Clearly, $l \leq m.$ Then,
 \begin{enumerate}
 \item  For any two $\mf v, \mf{v'} \in V_m,$ if $\mf u \in {\rm Chi(} \mf v)$ there exists unique $\mf u' \in {\rm Chi(}\mf v') $ such that $u_j - v_j  = u'_j - v'_j$, for $j=1,\ldots,n.$ Denote $k_j = u_j - v_j$ and $\bl k=(k_1,\ldots,k_n).$ Since for any other $\mf{v'} \in V_m$ there exists $\mf{u'} \in {\rm Chi(}\mf v')$ such that $u_j - v_j = u'_j - v'_j$, it follows that $k_j = u_j - v_j$ is independent of $v_j$. Thus corresponding to each element $\mf u_i \in {\rm Chi(}\mf v),$ we get a $\bl k^{(i)},$ for $i=1,\ldots,m.$  Clearly, there exist $m$ number of $n$-tuples which are associated to the graph $\m G$ and the association is unique by definition. Refer the set $\{\bl k^{(i)} :i=1,\ldots,m\}$ as deg($\m G$).
     \item Consider $\mf{u} \in V$. If $\ca(\text{\textnormal{Chi}}(\mf{u})) \leq m,$ then for each $\bl w \in \text{\textnormal{Chi}}(\mf{u}),$ there exists a unique $\bl k \in $ deg($\m G$) such that $w_j -u_j = k_j
     $ for all $j \in \{1,2,\ldots,n\}.$
     \end{enumerate}
 \end{enumerate}
Note that property (a) describes a condition on the set $V_m = \{\mf v \in V: \ca({\rm Chi}(\mf v)) = m\}$. However, property (b) describes a condition for every $u \in V$ with $\ca(\text{\textnormal{Chi}}(\mf{u})) \leq m.$

Let $\mathscr X(V) $ be the set of all generalized directed semi-trees on $V$ with Property \ref{10} and Property \ref{20}. 
For $V=\{(n,0): n \in \mb N_0\}$, the following graph yields an example of an element in $\mathscr X(V)$.

\includegraphics[scale=0.90]{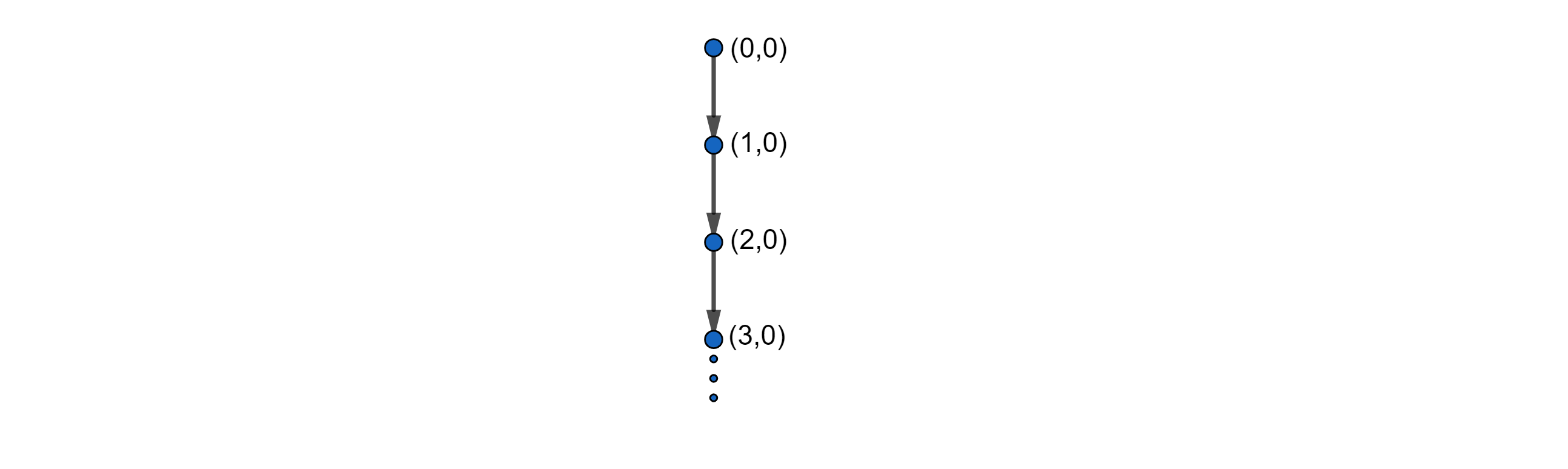}

For every $(\m G, l) \in \mathscr{X}(V)$, a unique polynomial $$p_{\m G} (z_1,\ldots,z_n) := \sum_{i=1}^m z_1^{k_1^{(i)}}\cdots z_n^{k_n^{(i)}},$$ 
is associated to the graph $(\m G,l)$. 
 
\subsubsection{Construction.} For $1 \leq i \leq n,$ let $(\mathcal{G}^{(i)}_n,m_i) = (V, E^{(i)}_n)$ be generalized directed semi-trees in $\mathscr{X}(V)$. Let $\text{\textnormal{Chi}}^{(i)}_n(\mf{v}) = \{\mf{u} : (\mf{v},\mf{u}) \in E^{(i)}_n\}$ and $\text{\textnormal{Par}}^{(i)}_n(\mf{u}) = \{\mf{v} : (\mf{v},\mf{u}) \in E^{(i)}_n\}.$

For $1 \leq i \leq n$, denote the unique polynomial associated to the graph $(\mathcal{G}^{(i)}_n,m_i)$ by $p_i$. Let $\{S_{\mf{v}}\}_{\mf{v} \in V}$ be a collection of polynomials in $n$ variables $p_1, p_2, \ldots, p_n$ which satisfy the following relations:

\begin{enumerate}[label=(\roman*)]
 \item \label{i} $S_{\mf{v}} = \displaystyle \sum_{|\bl{a}| \leq |\mf{v}|, a \in \mb N_0^n} c_{\bl{a}} p_1^{a_1} p_2^{a_2}\cdots p_n^{a_n} $ for some $c_{\bl{a}} \in \mb C$ and, 
 \item \label{ii} $p_i\cdot S_{\mf{v}} = \displaystyle \sum_{\mf{u} \in \text{\textnormal{Chi}}^{(i)}_n(\mf{v})} S_{\mf{u}},$
\end{enumerate} 
where $\bl a =(a_1,\ldots,a_n)$ is a tuple of non-negative integers and $|\bl{a}| = \sum_{i=1}^n{a_i}.$
Consider the space of formal power series $\mathscr H^2(\beta) = \{ \sum_{v \in V } \hat{f}(\mf{v})S_{\mf{v}} : \sum_{\mf{v} \in V} \lvert \hat{f}(\mf{v})\rvert^2 \beta(\mf{v})^2 < \infty \}$ with the inner product
\Bea
\langle f,g \rangle = \sum_{\mf{v} \in V}\hat{f}(\mf{v})\overline{\hat{g}(\mf{v})} \beta(\mf{v})^2 \text{~~for~~} f,g \in \mathscr H^2(\beta),
\Eea
 where $f=\sum_{v \in V } \hat{f}(\mf{v})S_{\mf{v}}$ and $g=\sum_{v \in V } \hat{g}(\mf{v})S_{\mf{v}}$ are formal power series and $\{ \beta (\mf{v}) : \mf{v} \in V\}$ is a net of positive numbers. Note that we have
\bea\label{beta}\langle S_{\mf{u}},S_{\mf{v}} \rangle = \begin{cases}
      0 & \textrm{if   } \mf{u} \neq \mf{v}, \\
      \beta(\mf{v})^2 & \textrm{if   }\mf{u} = \mf{v}. \\
      \end{cases}
\eea
Hence, the set $\{\frac{1}{\beta(\mf{v})}S_{\mf{v}}\}_{\mf{v} \in V}$ forms an orthonormal basis for $\mathscr H^2(\beta).$ Consider the $n$-tuple of multiplication operators $T = (T_{p_1},T_{p_2}, \cdots, T_{p_n})$ on the Hilbert space $\mathscr H^2(\beta)$ where
 $T_{p_i} : \mathscr H^2(\beta) \lo \mathscr H^2(\beta) $ is defined by \bea \label{mult} T_{p_i}f= p_i f,\,\,f \in \mathscr H^2(\beta).\eea
A priori it is not guaranteed that the operator $T_{p_i}$ is bounded. Following lemma describes a condition of boundedness of $T_{p_i}$.   
 
\begin{lem}\label{con}
For each $1 \leq i \leq n$, let $(\mathcal{G}^{(i)}_n,m_i)$ be a generalized semi-tree with the property $\ca({\rm Par}_n^{(i)}(\mf v))\leq k$ for $\mf v \in V$ and the unique polynomial associated to the graph $(\mathcal{G}^{(i)}_n,m_i)$ be $p_i$. The operator $T_{p_i}$ is bounded on $\mathscr{H}^2(\beta)$ if and only if $\text{\rm sup}_{\mf{v} \in V} \sum_{\mf{u} \in \text{\textnormal{Chi}}^{(i)}_n(\mf{v})} \frac{\beta(\mf{u})^2}{\beta(\mf{v})^2} < \infty $ for $i =1,2, \ldots, n.$
\end{lem}

\begin{proof} If each $T_{p_i}$ is bounded for $i=1,\ldots,n,$ then there exists $c>0$ such that $\lVert T_{p_i} S_{\mf v} \rVert^2 \leq c\lVert  S_{\mf{v}} \rVert^2$ for all $\mf v \in V.$ For $\mf v \in V,$
 \Bea
 \lVert T_{p_i} S_{\mf{v}} \rVert^2 &\leq& c\lVert  S_{\mf{v}} \rVert^2\\
   \lVert \sum_{\mf{u} \in \text{\textnormal{Chi}}^{(i)}_n(\mf{v})} S_\mf{u} \rVert^2  &\leq& c \beta({\mf{v}})^2\\
    \sum_{\mf{u} \in \text{\textnormal{Chi}}^{(i)}_n(\mf{v})} \lVert S_\mf{u} \rVert^2 &\leq& c \beta({\mf{v}})^2\\
    \sum_{\mf{u} \in \text{\textnormal{Chi}}^{(i)}_n(\mf{v})} \frac{\beta({\mf{u}})^2}{\beta({\mf{v}})^2} &\leq& c.
 \Eea
 Conversely, assume that $\text{sup}_{\mf{v} \in V} \sum_{\mf{u} \in \text{\textnormal{Chi}}^{(i)}_n(\mf{v})} \frac{\beta(\mf{u})^2}{\beta(\mf{v})^2} = c < \infty.$ Let $f \in\mathscr{H}^2(\beta)$ be such that $f = \sum_{\mf u \in W} f_{\mf u} S_{\mf u},$ where $W$ is a finite subset of $V$. It follows from Property \ref{i} and Equation \eqref{mult} that
 \bea T_{p_i} f = \sum_{\mf v \in W} \big(\sum_{\mf u \in {\rm Par}_n^{(i)}(\mf v)} \hat{f}(\mf u)\big) S_{\mf v}, \eea for $f \in \mathscr H^2(\beta).$
 Hence, 
 \Bea \norm{T_{p_i} f}^2 
 &=& \sum_{\mf v \in W} |\sum_{\mf u \in {\rm Par}_n^{(i)}(\mf v)} \hat{f}(\mf u)|^2 \beta(\mf v)^2 \\ 
 &\leq& \sum_{\mf v \in W} \big(\sum_{\mf u \in {\rm Par}^{(i)}(\mf v)} |\hat{f}(\mf u)|^2\big)\big(\sum_{\mf u \in {\rm Par}^{(i)}(\mf v)} 1^2\big) \beta(\mf v)^2\\
 &\leq& \sum_{\mf v \in W} \big(k\sum_{\mf u \in {\rm Par}^{(i)}(\mf v)} |\hat{f}(\mf u)|^2\big) \beta(\mf v)^2\\ 
 &=& k\sum_{\mf v \in W} \big(\sum_{\mf u \in {\rm Chi}^{(i)}(\mf v)} \beta(\mf u)^2 \big)  |\hat{f}(\mf v)|^2 \\ &=& 
 k\sum_{\mf v \in W} \big(\sum_{\mf u \in {\rm Chi}^{(i)}(\mf v)} \frac{\beta(\mf u)^2}{\beta(\mf v)^2} \big)  |\hat{f}(\mf v)|^2 \beta(\mf v)^2 \leq kc \norm{f}^2. \Eea
Since $\ca({\rm Par}_n^{(i)}(\mf v)) \leq k$ for every $\mf v \in V$, the second inequality follows. Thus the operator $T_{p_i}$ is bounded on a dense subset of $\mathscr{H}^2(\beta)$ and therefore $T_{p_i}$ is bounded.
\end{proof}
A straightforward calculation yields the following lemma.
\begin{lem}
Suppose that $(T_{p_1},\ldots,T_{p_n})$ is an $n$-tuple of bounded operators on the Hilbert space $\mathscr{H}^2(\beta).$ Then
\[ T^*_{p_i} S_{\mf{v}} = \begin{cases}
      0, & \textrm{if   } \mf{v} \in \text{\rm Root}(\mathcal G^{(i)}_n), \\
      \sum_{\mf{u} \in \text{\textnormal{Par}}^{(i)}_n(\mf{v})} \frac{\beta({\mf{v}})^2}{\beta({\mf{u}})^2}S_{\mf{u}}, & \textrm{otherwise.   } \\

   \end{cases}
\]
\end{lem}


\begin{defn}\label{bpe1} For $\bl w \in \mb C^n,$ let $\l_{\bl w}$ be the linear functional on the subspace $\bigvee\{S_{\mf v}: \mf v \in V\}$ of linear span of $S_{\mf v}$, defined by $\l_{\bl w}(f) = f(\bl w),\,\,f \in \bigvee \{S_{\mf v}: \mf v \in V\}$. The point $\bl w$ is said to be a bounded point evaluation on $\mathscr H^2(\beta)$ if the functional $\l_{\bl w}$ on $\bigvee\{S_{\mf v}: \mf v \in V\}$ extends to a bounded linear functional on $\mathscr H^2(\beta).$
\end{defn}

\begin{lem}\label{bpe}
  $\bl w \in \mb C^n$ is a bounded point evaluation if and only if $\sum_{\mf v \in V} \frac{|  S_{\mf v}(\bl{w})|^2}{\beta(\mf v)^2} < \infty.$ 
\end{lem}

\begin{proof}
From Definition \ref{bpe1}, it follows that if $\bl w$ is a bounded point evaluation, then $ \l_{\bl w}(p) =  p(\bl w)$ for all polynomials $p$ in $p_i$ variables. By Riesz representation theorem, if $\bl w$ is a bounded point evaluation, then there exists $\gamma_{\bl w} \in \mathscr H^2(\beta)$ such that $\l_{\bl w}(f) = \langle f, \gamma_{\bl w} \rangle.$ Therefore we have $S_{\mf v}(\bl w) =  \langle S_{\mf v}, \gamma_{\bl w} \rangle= \overline{\hat{\gamma}_{\bl w}(\mf{v})} \beta(\mf v)^2,$ that is, $\hat{\gamma}_{\bl w}(\mf{v}) = \frac{\overline{ S_{\mf v}(\bl w)}}{\beta(\mf v)^2}.$ Since $\gamma_{\bl w} \in \mathscr H^2(\beta),$ it follows that $\norm{\gamma_{\bl w}}^2 = \sum_{\mf v \in V} |\hat{\gamma}_{\bl w}(\mf v)|^2 \beta(\mf v)^2 = \sum_{\mf v \in V} \frac{| S_{\mf v}(\bl w)|^2}{\beta(\mf v)^2} < \infty.$

Conversely, suppose that for $\bl w \in \mb C^n,$ we have $\sum_{\mf v \in V} \frac{|  S_{\mf v}(\bl{w})|^2}{\beta(\mf v)^2} =c.$ Then, the formal power series $\gamma_{\bl w} = \sum_{\mf v \in V} \frac{\overline{ S_{\mf v}(\bl w)}}{\beta(\mf v)^2} S_{\mf v}$ is in $\mathscr H^2(\beta).$ Let us define $\l_{\bl w} : \mathscr{H}^2(\beta) \to \C$ such that $\l_{\bl w}(f) = \inner{f}{\gamma_{\bl w}}.$ A direct computation gives us $\l_{\bl w}(S_{\mf v}) = S_{\mf v}(\bl w)$ and thus $\l_{\bl w}(g) = g(\bl w)$ for every $g \in \bigvee \{S_{\mf v} : \mf v \in V\}$. Therefore, $\bl w$ is a bounded point evaluation.
\end{proof}


\begin{prop}\label{function space}
If $\bl w$ is a bounded point evaluation and $f = \sum_{\mf v \in V} \hat{f}(\mf v)S_{\mf v} \in \mathscr{H}^2(\beta),$ then the series $\sum_{\mf v \in V} \hat{f}(\mf v)S_{\mf v}(\bl w)$ converges absolutely to $\lambda_{\bl w}(f).$ 
\end{prop}
\begin{proof}
Since $w$ is a bounded point evaluation, Lemma \ref{bpe} yields that $ \sum_{\mf v \in V}\frac{|S_{\mf v}(\bl w)|^2}{\beta(\mf v)^2} < \infty$. Let $ \sum_{\mf v \in V}\frac{|S_{\mf v}(\bl w)|^2}{\beta(\mf v)^2} = c.$ Suppose $U$ is a finite subset of $V.$ We have 
\Bea
\left(\sum_{\mf v \in U} |\hat{f}(\mf v)S_{\mf v}(\bl w)|\right)^2 &=& \left(\sum_{\mf v \in U} |\hat{f}(\mf v)\beta(\mf v) \frac{S_{\mf v}(\bl w)}{\beta(\mf v)}|\right)^2\\
&\leq& \sum_{\mf v \in U} |\hat{f}(\mf v)|^2\beta(\mf v)^2 \sum_{\mf v \in U}\frac{|S_{\mf v}(\bl w)|^2}{\beta(\mf v)^2}\\ &\leq& \sum_{\mf v \in V} |\hat{f}(\mf v)|^2\beta(\mf v)^2 \sum_{\mf v \in V}\frac{|S_{\mf v}(\bl w)|^2}{\beta(\mf v)^2}\\
&=& \norm{f}^2\sum_{\mf v \in V}\frac{|S_{\mf v}(\bl w)|^2}{\beta(\mf v)^2}\\
&=& \norm{f}^2c.
\Eea
This proves that the series $\sum_{\mf v \in V} \hat{f}(\mf v)S_{\mf v}(\bl w)$ converges absolutely. Since we have $\lambda_{\bl w} \left(\sum_{\mf v \in U} \hat{f}(\mf v)S_{\mf v} \right) = \sum_{\mf v \in U} \hat{f}(\mf v)S_{\mf v}(\bl w)$ for every finite subset $U$ of $V$ and $\lambda_{\bl w}$ is continuous, it follows that the series $\sum_{\mf v \in V} \hat{f}(\mf v)S_{\mf v}(\bl w)$ converges absolutely to $\lambda_{\bl w}(f).$ 
\end{proof}
\begin{rem}
In view of Proposition \ref{function space}, there is no ambiguity in writing $\lambda_{\bl w}(f) = f(\bl w),$ whenever $\bl w$ is a bounded point evaluation. 
    \end{rem}
The set of all bounded point evaluations on $\mathscr H^2(\beta)$ is denoted by $\Omega_{bpe}.$ Then the following corollary follows from Proposition \ref{function space}.
\begin{cor}\label{cor}
The Hilbert space $\mathscr{H}^2(\beta)$ is a reproducing kernel Hilbert space with reproducing kernel
$$\kappa(z, w) = \displaystyle \sum_{\mf v \in V} \frac{S_{\mf v}(z) \overbar{S_{\mf v}(w)}}{\beta(\mf v)^2},\,\,z, w \in \Omega_{bpe},$$
consisting of holomorphic functions on $\Omega_{bpe}$ if the interior of $\Omega_{bpe}$ is non-empty.
\end{cor}
There are natural examples where $\Omega_{bpe}$ has non-empty interior. We have provided two  such classes of examples in Example I in Section \ref{example}. 

For $\bl w \in \mb C^n,$ let $p_{\bl w} = (p_1(\bl w),p_2(\bl w),\ldots, p_n(\bl w)).$ Then the following result holds.
\begin{prop}
Suppose that $(T_{p_1},\ldots,T_{p_n})$ is an $n$-tuple of bounded operators on $\mathscr{H}^2(\beta).$ Then $\overline{p_{\bl w}}$ is in the point spectrum
 $\sigma(T^*_{p_1},\ldots,T^*_{p_n})$ with common eigenvector $K_{\bl w}$ if $\bl w$ is a bounded point evaluation on $\mathscr H^2(\beta)$.
\end{prop}

\begin{proof}
Suppose $\bl w$ is a bounded point evaluation on $\mathscr H^2(\beta).$ By Riesz representation theorem, there exists $K_{\bl w} \in \mathscr H^2(\beta)$ such that $\l_{\bl w}(f) = \langle f, K_{\bl w} \rangle$ for $f \in \mathscr H^2(\beta)$. Then we have,
\Bea
\langle f, T^*_{p_i} K_{\bl w} \rangle &=& \langle T_{p_i}f,  K_{\bl w} \rangle \\
&=& \langle p_i f,  K_{\bl w} \rangle\\
&=& p_i(\bl w) \langle f,  K_{\bl w} \rangle\\
&=&  \langle f, \overline{p_i(\bl w)} K_{\bl w}. \rangle
\Eea
Therefore $T^*_{p_i} K_{\bl w} = \overline{p_i(\bl w)} K_{\bl w}$ and hence, $\overline{p_{\bl w}} \in \sigma(T^*_{p_1},T^*_{p_2},\ldots,T^*_{p_n})$ corresponding to common eigenvector $K_{\bl w}.$
\end{proof}
\subsubsection{Relation Between Analytic Structure and Discrete Structure}

Let $\{(\mathcal{G}_n^{(i)}, m_i) = (V, E_n^{(i)}) : i=1,\ldots,n\} \subset \mathscr X(V)$ and the unique polynomial associated to the graph $(\mathcal{G}_n^{(i)}, m_i)$ be $p_i.$  Suppose that there exists $k>0$ such that $\ca({\rm Par}_n^{(i)}(\mf v))\leq k$ for all $v \in V$ and $i=1,\ldots,n.$ For each $1 \leq i \leq n$, the weighted shift $\Lambda_i$ on $(\mathcal{G}^{(i)}_n,m_i)$ is defined as following: for $f \in \mathcal D(\Lambda_i) \subseteq \ell^2(V)$
\bea\label{eq: 0.1}(\Lambda_i f)(\mf v) =
 \begin{cases}
 0, & \text{~if~} \mf v = \text{Root}(\mathcal{G}^{(i)}_n),\\
      \sum_{u \in \textnormal{Par}_n^{(i)}(\mf v)} \frac{\beta(\mf u)}{\beta(\mf v)} f(\mf u),    &  \text{~otherwise,}
   \end{cases}\eea
   where $\beta(\mf v)$ is as in Equation \eqref{beta}.
\begin{thm}\label{lem}
Let $\Lambda_i$ be the weighted shifts on the generalized directed semi-tree $(\mathcal G_i,m_i),$ as described in Equation \eqref{eq: 0.1}, for $i=1,\ldots,n$ such that each $\Lambda_i$ extends to a bounded linear operator to $\ell^2(V)$ and the n-tuple $(\Lambda_1, \ldots, \Lambda_n)$ is commuting. There exists a unitary operator $U : \ell^2(V) \to \mathscr H^2(\beta)$ such that $U^* T_{p_i} U = \Lambda_i$ for each $1 \leq i \leq n$. 
\end{thm}

\begin{proof}
For $\mf u \in V$, let $\chi_{\mf u} \in \ell^2(V)$ be given by the Equation \eqref{eqn w1}. Note that the sets $\{\chi_{\mf u} : \mf u \in V\}$ and $\left\lbrace \frac{S_{\mf u}}{\beta(\mf u)} : u \in V \right\rbrace$ are orthonormal basis of $\ell^2(V)$ and $\mathscr{H}^2(\beta)$, respectively. Let $U : \ell^2(V) \to \mathscr H^2(\beta)$ be defined by $U \chi_{\mf u} = \frac{S_{\mf u}}{\beta(\mf u)}$. It follows from the Equation \eqref{chi} and Equation \eqref{eq: 0.1} that 
\begin{equation}\label{eq:0.2}
\Lambda_i \chi_{\mf v} = \sum_{\mf u \in \textnormal{Chi}_n^{(i)}(\mf v)} \frac{\beta(\mf u)}{\beta(\mf v)} \chi_{\mf u},\,\,\mf v \in V,
\end{equation}
that is,
\Bea
(\Lambda_i\chi_{\mf v})(\mf u) =
\begin{cases}
     \frac{\beta(\mf u)}{\beta(\mf v)},& \text{~if~} \mf u \in {\rm Chi}_n^{(i)}(\mf v),\\
     0,& \text{~otherwise.~}
\end{cases}
\Eea
From the property $\ref{ii}$ of the collection of polynomials $\{S_{\mf v}(\bl z)\}_{\mf v \in V}$, we have
\begin{equation}\label{eq:0.3}
T_{p_i} \frac{S_{\mf v}}{\beta(\mf v)} = \sum_{\mf u \in \textnormal{Chi}_n^{(i)}(\mf v)} \frac{\beta(\mf u)}{\beta(\mf v)} \frac{S_{\mf u}}{\beta(\mf u)},\,\,\mf v \in V.
\end{equation}
Therefore, from the Equation \eqref{eq:0.2} and Equation \eqref{eq:0.3}, it follows that $U^* T_{p_i} U = \Lambda_i$.
\end{proof}


\section{Example I}\label{example}
Fix $n >1.$ Recall that the weighted Bergman space $\mb A^{(\l)}(\mb D^n)$, $\lambda >0,$ consisting of holomorphic functions on $\mathbb D^n$, is determined by the reproducing kernel $K^{(\l)}:\mb D^n\times\mb D^n\to\mb{C}$ given by the formula
\Bea
K^{(\l)}(\bl z, \bl w)=\prod_{j=1}^n(1-z_j\bar w_j)^{-\l},\,\, \bl z,\bl w\in\mb D^n.
\Eea
Let $s_k$ denote the elementary symmetric polynomial of degree $k$ in $n$ variables, that is, $$s_k(z_1,\ldots,z_n) = \sum_{1\leq i_1< i_2 <\ldots <i_k\leq n} z_{i_1} \cdots z_{i_k}.$$ The map $\bl s=(s_1,\ldots,s_n) : \mb C^n \to \mb C^n$ is called the symmetrization map.
For $1 \leq i \leq n$, $M_{s_i}$ denotes the multiplication operator by the elementary symmetric polynomial $s_i$ on $\mb A^{(\lambda)}(\mb D^n)$ and $M_{\bl s}:=(M_{s_1},\ldots,M_{s_n})$. The permutation group on $n$ symbols is denoted by $\mathfrak S_n.$
The subspaces \Bea
\mb A_{\rm sym}^{(\l)}(\mb D^n) =\{f\in \mb A^{(\l)}(\mb D^n):f\circ\sigma^{-1}=f \mbox{~for~} \sigma\in\mathfrak S_n\},
\Eea and
\Bea
\mb A_{\rm anti}^{(\l)}(\mb D^n)=\{f\in \mb A^{(\l)}(\mb D^n):f\circ\sigma^{-1}={\rm sgn} (\sigma)f \mbox{~for~} \sigma\in\mathfrak S_n\}
\Eea are two joint reducing subspaces of the $n$-tuple of multiplication operators $M_{\bl s}$ on $\mb A^{(\lambda)}(\mb D^n),$ see \cite[p. 774]{MR3906291}. In the following discussion, the restriction operators $M_{\bl s}|_{\mb A_{\rm anti}^{(\l)}(\mb D^n)}$ and $M_{\bl s}|_{\mb A_{\rm sym}^{(\l)}(\mb D^n)}$ have been identified with two $n$-tuple of weighted shift operators on generalized directed semi-trees. 

 The collection of all elements $\bl u = (u_1,u_2,...,u_n) \in \mb N_0^n \text{  with  } u_1 \geq u_2 \geq \cdots \geq u_n \geq 0 $ is denoted by $\m Q_n.$  For $\bl u \in \m Q_n,$ let $(\l)_{\bl u}=\prod_{j=1}^n(\l)_{u_j}$, where $(\l)_{u_j}=\l(\l+1)\ldots (\l+u_j-1)$ is the Pochhammer symbol.

For $\bl u \in \m Q_n$, suppose $a_{\bl u }(\bl z)=\det(\!(z_j^{u_i})\!)_{i,j=1}^{n}$.  Let $\m P^{(\rm anti)}_n = \{\bl v + \bl \delta: \bl v \in \m Q_n \text{~and~} \bl \delta = (n-1,n-2,\ldots, 1,0)\}.$ For $\bl u \in \m P^{(\rm anti)}_n,$ the norm of $a_{\bl u}$ in $ \mb{A}^{(\l)} (\mb{D}^n)$ is calculated in \cite[p. 775]{MR3906291}. It is given by \bea\label{weight}\norm{a_{\bl u}} =\sqrt{\frac{\bl u! n!}{(\l)_{\bl u}}}= \frac{1}{\g_{\bl u}} \text{~(say)},\eea where $\bl u!=\prod_{j=1}^n u_j!$. Moreover, the following holds \cite[p. 2365]{MR3043017}:
 \begin{lem}
 The set $\{\g_{\bl u}a_{\bl u}:\bl u\in \m P^{(\rm anti)}_n\}$ forms an orthonormal basis for $\mb A_{\rm anti}^{(\l)}(\mb D^n)$.
 \end{lem}
The action of $\mathfrak{S}_n$ on $\mb N_0^n$ is given by
$(\sigma,\bl u)\mapsto \sigma\cdot \bl u=(u_{\sigma^{-1}(1)},\ldots, u_{\sigma^{-1}(n)}). $ Let $\mathfrak S_n\bl u$ denote the orbit of $\bl u\in  \mb N_0^n.$ We say $\bl u_1 \sim \bl u_2$ if $\bl u_1 = \sigma \cdot \bl u_2$ for some $\sigma \in \mathfrak{S}_n.$ For an element $\bl u \in \m Q_n,$ let $[\bl u] = \{\bl v \in \m Q_n: \bl v \sim \bl u \}$ be the equivalence class and $\m P^{(\rm sym)}_n = \{[\bl u]:\bl u \in \m Q_n  \}.$ 

For $\bl u\in \m Q_n,$ consider the  monomial symmetric polynomials (as in \cite[p. 454]{MR1153249})
 $$m_{\bl u}(\bl z) = \sum_{\bl\beta \in \mathfrak S_n \bl u}\bl z^{\bl\beta},$$
 where the sum is over all distinct permutations $\bl \beta=(\beta_1,\beta_2,...,\beta_n)$ of $\bl u$ and $\bl z^{\bl\beta} = z_1^{\beta_1} z_2^{\beta_2}...z_n^{\beta_n}.$ 
Observe that $m_{{\bl u}}(\bl z)=\sum_{\bl\b}\bl z^{\bl \b} = m_{{\bl u}^\i}(\bl z)\mbox{~for~}{\bl u}^{\i} \in [\bl u]$ and the elements of the set $\{m_{\bl u}:\bl u\in \m P^{(\rm sym)}_n\}$ are mutually orthogonal in $\mb A^{(\l)}(\mb D^n)$. 
 
 If  $\bl u\in \m Q_n$ has $k (\leq n)$ distinct components, that is, there are $k$ distinct  non-negative integers $u_1>\ldots> u_k$ such that
\Bea
\bl u =( u_1,\ldots,u_1, u_2,\ldots,u_2,\ldots,u_k,\ldots,u_k),
\Eea
where each $u_i$ is repeated $\a_i$ times, for $i=1,\ldots, k,$ then $\bl \a=(\a_1,\ldots, \a_k)$ is said to be the  {\it multiplicity} of $\bl u.$ For a fixed $\bl u\in\m Q_n$ with multiplicity $\bl \alpha$, we have $\ca( \mathfrak  S_n\bl u)=\frac{n!}{\bl\a!}.$ For an element $\bl u \in \m P^{(\rm sym)}_n$ with multiplicity $\bl \alpha,$ the norm of $m_{\bl u}$ in $\mb A^{(\l)}(\mb D^n)$ can be calculated by the orthogonality of distinct monomials in $\mb A^{(\l)}(\mb D^n)$ and it is given by  \bea\label{weight1}\norm{m_{\bl u}} = \sqrt{\frac{\bl u!n!}{(\l)_{\bl u}\bl \alpha!}} = \frac{1}{k_{\bl u}} \text{~(say)}.\eea Moreover, from the discussion in \cite[Section 4]{MR3906291}, we get the following lemma.
\begin{lem}
The set $\{k_{\bl u}m_{\bl u}:\bl u\in \m P^{(\rm sym)}_n\}$ forms an orthonormal basis for $\mb A_{\rm sym}^{(\l)}(\mb D^n)$.
\end{lem}

For each partition $\bl u \in \m Q_n$, the associated Schur polynomial in $n$ variables is defined by \bea\label{Schur}S_{\bl u}(z_1,z_2,...,z_n) &=& \frac{\det(\!(z_j^{u_i+n-i})\!)_{i,j=1}^{n}}{\det(\!(z_j^{n-i})\!)_{i,j=1}^{n}}\\
&=& \nonumber \frac{a_{\bl u + \bl \delta}(\bl z)}{a_{\bl \delta}(\bl z)},\eea \cite[p. 454, A.4]{MR1153249}. 
For $1\leq k\leq n,$ the partition with the first $k$ components as $1$ and the rest $(n-k)$ components as $0$, is denoted by $(\bl 1^k).$ Observe from the second of {\it Giambelli's formulas}, stated in \cite[p. 455]{MR1153249}, that $S_{(\bl 1^k)}= s_k.$ By Pieri formula in \cite[p. 25]{MR1464693}, we have \bea\label{Pieri} s_k S_{\bl u}  = \sum_{\bl v \in \m I^{(k)}_{\bl u}} S_{\bl v},\eea 
where $\m I^{(k)}_{\bl u} = \{\bl v = (v_1,\ldots,v_n) \in \m Q_n : \sum_{j=1}^n (v_j-u_j) =  k \text{~and~} 0 \le v_j-u_j \leq 1 \text{~for all~} j \}.$ Equivalently, $s_k a_{\bl u +\bl \delta}  = \sum_{\bl v \in \m I^{(k)}_{\bl u} } a_{\bl v +\bl \delta}.$ Then it follows that for $\bl u \in \m P^{(\rm anti)}_n$, 
\bea
\label{basis}s_k a_{\bl u }  = \sum_{\bl v \in \m I^{(k)}_{\bl u} \cap \m P^{(\rm anti)}_n} a_{\bl v }.
\eea
A straightforward calculation shows that the monomial symmetric polynomials follow a similar summation formula as the Schur polynomials. That is, for $\bl u \in \m P^{(\rm sym)}_n,$ \bea\label{monomial} s_k m_{\bl u}  = \sum_{\bl v \in (\m I^{(k)}_{\bl u}/ \sim) \cap \m P^{(\rm sym)}_n} c_{\bl v} m_{\bl v},\eea
where $c_{\bl v}$ are positive rational numbers.



\begin{rem}\label{card}
Let $C_{n,k}$ be the collection of all $(t_1,\ldots,t_n) \in \mb N^n_0$ such that exactly $k$ number of $t_i$'s are equal to $1$ and the rest $t_i$'s are $0$. Since we have exactly $\binom{n}{k}$ ways to choose such a $(t_1,\ldots,t_n),$ so the cardinality of $C_{n,k}$ is $\binom{n}{k}.$ We note that if $\bl v \in \m I^{(k)}_{\bl u}$, then $\bl v = \bl u + \bl t$ for some $\bl t \in C_{n,k}$. Therefore, for any $\bl u$ cardinality of $\m I^{(k)}_{\bl u}$ can at most be $\binom{n}{k}.$
\end{rem}
For each $1\leq k \leq n,$ we define the directed graph $G_{\rm anti}^{(k)}= (\m P^{(\rm anti)}_n, E_{\rm anti}^{(k)}),$ where the set of all edges is given by $E_{\rm anti}^{(k)} =\bigcup_{\bl u \in \m P^{(\rm anti)}_n}\{(\bl u,\bl v): \bl v \in \m I^{(k)}_{\bl u} \cap \m P^{(\rm anti)}_n\}.$
\begin{prop}\label{semtres}
For each $1\leq k \leq n$, there exists a non-negative integer $m_k \leq \binom{n}{k}$ such that the directed graph $\big(G_{\rm anti}^{(k)},m_k\big)$ is a generalized directed semi-tree.  
\end{prop}
\begin{proof}
Fix $1 \leq k \leq n.$ Since $\m P^{(\rm anti)}_n$ is countable, so the graph $G_{\rm anti}^{(k)}$ can have at most countably many connected components. Since Root$\big(G_{\rm anti}^{(k)}\big) \subseteq \m P^{(\rm anti)}_n,$ the graph can have at most countable roots.

Moreover, for every $\bl u \in \m P^{(\rm anti)}_n$ we have ${\rm Chi}^{(k)}_n(\bl u)=\m I^{(k)}_{\bl u} \cap \m P^{(\rm anti)}_n,$ where the set of children of $\bl u$ in $G_{\rm anti}^{(k)}$ is denoted by ${\rm Chi}^{(k)}_n(\bl u).$ By the Remark \ref{card}, it is evident that $\ca(\m I^{(k)}_{\bl u})  \leq \binom{n}{k}$ and thus $\ca(\m I^{(k)}_{\bl u} \cap \m P^{(\rm anti)}_n) \leq \ca(\m I^{(k)}_{\bl u})  \leq \binom{n}{k}.$ Therefore, for any two distinct $\bl u, \bl v \in \m P^{(\rm anti)}_n,$ $\ca({\rm Chi}^{(k)}_n(\bl u) \cap {\rm Chi}^{(k)}_n(\bl v)) \leq \binom{n}{k}.$ Consequently, we have $m_k = {\rm sup}_{\bl u, \bl v \in \m P^{(\rm anti)}_n, \bl u \neq \bl v} \ca({\rm Chi}^{(k)}_n(\bl u) \cap {\rm Chi}^{(k)}_n(\bl v)).$  

To prove that $G_{\rm anti}^{(k)}$ has no circuit, we argue by contradiction. Suppose not, then there exists a sequence $\{ \bl v_1,\ldots,\bl v_n\}, n>1$ such that $(\bl u,\bl v_1),(\bl v_n,\bl u)$ and $(\bl v_j,\bl v_{j+1}) \in E_{\rm anti}^{(k)}$ for all $j=1,\ldots,n-1.$ Since $(\bl u,\bl v_1) \in E_{\rm anti}^{(k)},$ by construction, there exists at least one $i=1,\ldots,n$ such that $u_i < {v_1}_i,$ where $\bl u=(u_1,\ldots,u_n)$ and $\bl v_1=({v_1}_1,\ldots,{v_1}_n).$ This implies $u_i< {v_1}_i \leq  {v_2}_i \leq {v_n}_i \leq u_i,$ which is a contradiction.
\end{proof}

In a similar manner, define the directed graph $G_{\rm sym}^{(k)}= (\m P^{(\rm sym)}_n, E_{\rm sym}^{(k)}),$ where  $E_{\rm sym}^{(k)} =\bigcup_{\bl u \in \m P^{(\rm sym)}_n}\{(\bl u,\bl v): \bl v \in (\m I^{(k)}_{\bl u}/ \sim) \cap \m P^{(\rm sym)}_n\},$ for each $1\leq k \leq n.$ The children of $\bl u$ in $E_{\rm sym}^{(k)}$ is denoted by ${\rm Chi}^{(k)}(\bl u).$ Suppose that $r_k = {\rm sup}_{\bl u, \bl v \in \m P^{(\rm sym)}_n, \bl u \neq \bl v} \ca({\rm Chi}^{(k)}(\bl u) \cap {\rm Chi}^{(k)}(\bl v)).$ Then similar arguments as above lead us to the following proposition. 
\begin{prop}\label{symchobi}
The directed graph $\big(G_{\rm sym}^{(k)},r_k\big)$ is a generalized directed semi-tree, where $r_k \leq \binom{n}{k}$ for each $1\leq k \leq n$.
\end{prop}

For $1\leq k \leq n$, let $\Lambda^{(\rm anti)}_k : \ell^2(\m P^{(\rm anti)}_n) \to \ell^2(\m P^{(\rm anti)}_n)$ be such that 
\begin{equation}\label{new x}
\Lambda^{(\rm anti)}_k \chi_{\bl u} = \sum_{\bl v \in \m I^{(k)}_{\bl u} \cap \m P^{(\rm anti)}_n} \frac{\gamma_{\bl v}}{\gamma_{\bl u}} \chi_{\bl v},
\end{equation} 
where $\gamma_{\bl u}$ is given by the Equation \eqref{weight}.
Therefore,  $\bl \Lambda^{(\rm anti)}=(\Lambda^{(\rm anti)}_1,\ldots,\Lambda^{(\rm anti)}_n)$ is an $n$-tuple of weighted shifts where each $\Lambda^{(\rm anti)}_k$ is a weighted shift on the generalized directed semi-tree $G_{\rm anti}^{(k)}$. 

Similarly, we define $\Lambda^{(\rm sym)}_k : \ell^2(\m P^{(\rm sym)}_n) \to \ell^2(\m P^{(\rm sym)}_n)$ by 
\begin{equation}\label{new xx}
\Lambda^{(\rm sym)}_k \chi_{\bl u} = \sum_{\bl v \in (\m I^{(k)}_{\bl u}/ \sim) \cap \m P^{(\rm sym)}_n } \frac{k_{\bl v}}{k_{\bl u}} c_{\bl v} \chi_{\bl v},
\end{equation} 
where $k_{\bl u}$ is given by the Equation \eqref{weight1}. The $n$-tuple of operators is denoted by
$\bl \Lambda^{(\rm sym)}=(\Lambda^{(\rm sym)}_1,\ldots,\Lambda^{(\rm sym)}_n),$ where each $\Lambda^{(\rm sym)}_k$ is a weighted shift on the generalized directed semi-tree $G_{\rm sym}^{(k)}.$

The operators $U^{(\rm sym)} : \ell^2(\m P^{(\rm sym)}_n) \to \mb A^{(\l)}_{\rm sym}(\mb D^n)$ and $U^{(\rm anti)} :\ell^2(\m P^{(\rm anti)}_n) \to \mb A^{(\l)}_{\rm anti}(\mb D^n)$ which defined by \Bea U^{(\rm sym)}\chi_{\bl u} &=& m_{\bl u},  \text{~for } \bl u \in \m P^{(\rm sym)}_n\text{~and},\\
 U^{(\rm anti)}\chi_{\bl u} &=& a_{\bl u}, \text{~for } \bl u \in \m P^{(\rm anti)}_n\Eea are unitary. From the definition, it follows that for every $k = 1,\ldots,n,$ the unitary $U^{(\rm anti)}$ intertwines the operators $\Lambda^{(\rm anti)}_k$ on $\ell^2(\m P^{(\rm anti)}_n)$ and $M_{s_k}$ on $\mb A^{(\l)}_{\rm anti}(\mb D^n)$  and $U^{(\rm sym)}$ intertwines the operators $\Lambda^{(\rm sym)}_k$ on $\ell^2(\m P^{(\rm sym)}_n)$ and $M_{s_k}$ on $\mb A^{(\l)}_{\rm sym}(\mb D^n).$ Therefore, we have the following results:
\begin{thm}
Suppose $\l >0.$ We have the following:
\begin{enumerate}
    \item the $n$-tuple of operators $(M_{s_1},\ldots,M_{s_n})$ on $\mb A_{\rm sym}^{(\l)}(\mb D^n)$ is unitarily equivalent to the $n$-tuple of operators $(\Lambda^{(\rm sym)}_1,\ldots,\Lambda^{(\rm sym)}_n)$ on $\ell^2(\m P^{(\rm sym)}_n),$ where $\Lambda^{(\rm sym)}_k$ is a weighted shift on the generalized directed semi-tree $\big(G_{\rm sym}^{(k)},r_k\big),$ defined in the Equation \eqref{new xx} and 
    \item the $n$-tuple $(M_{s_1},\ldots,M_{s_n})$ on $\mb A_{\rm anti}^{(\l)}(\mb D^n)$ is unitarily equivalent to the $n$-tuple of operators $(\Lambda^{(\rm anti)}_1,\ldots,\Lambda^{(\rm anti)}_n)$ on $\ell^2(\m P^{(\rm anti)}_n),$ where $\Lambda^{(\rm anti)}_k$ is a weighted shift on the generalized directed semi-tree $\big(G_{\rm anti}^{(k)},m_k\big),$ given in the Equation \eqref{new x}.
\end{enumerate}
\end{thm}
In addition to it, note that point spectra of  $M^*_{\bl s}|_{\mb A_{\rm anti}^{(\l)}(\mb D^n)}$ and $M^*_{\bl s}|_{\mb A_{\rm sym}^{(\l)}(\mb D^n)}$ contain the symmetrized polydisc $\mb G_n (= \bl s{(\mb D^n})).$ This provides two classes of examples which satisfy the hypothesis  in Corollary \ref{cor}.

\subsection{Pictorial representation for $G_{\rm sym}^{(k)}$}
The generalized directed semi-trees $G_{\rm sym}^{(1)},$ $G_{\rm sym}^{(2)}$ and $G_{\rm sym}^{(3)}$ below represent the multiplication operators $M_{s_1}, M_{s_2}$ and $M_{s_3}$ on $\mb A^{(\l)}_{\rm sym}(\mb D^3),$ respectively (cf. Proposition \ref{symchobi}).

\includegraphics[scale=1]{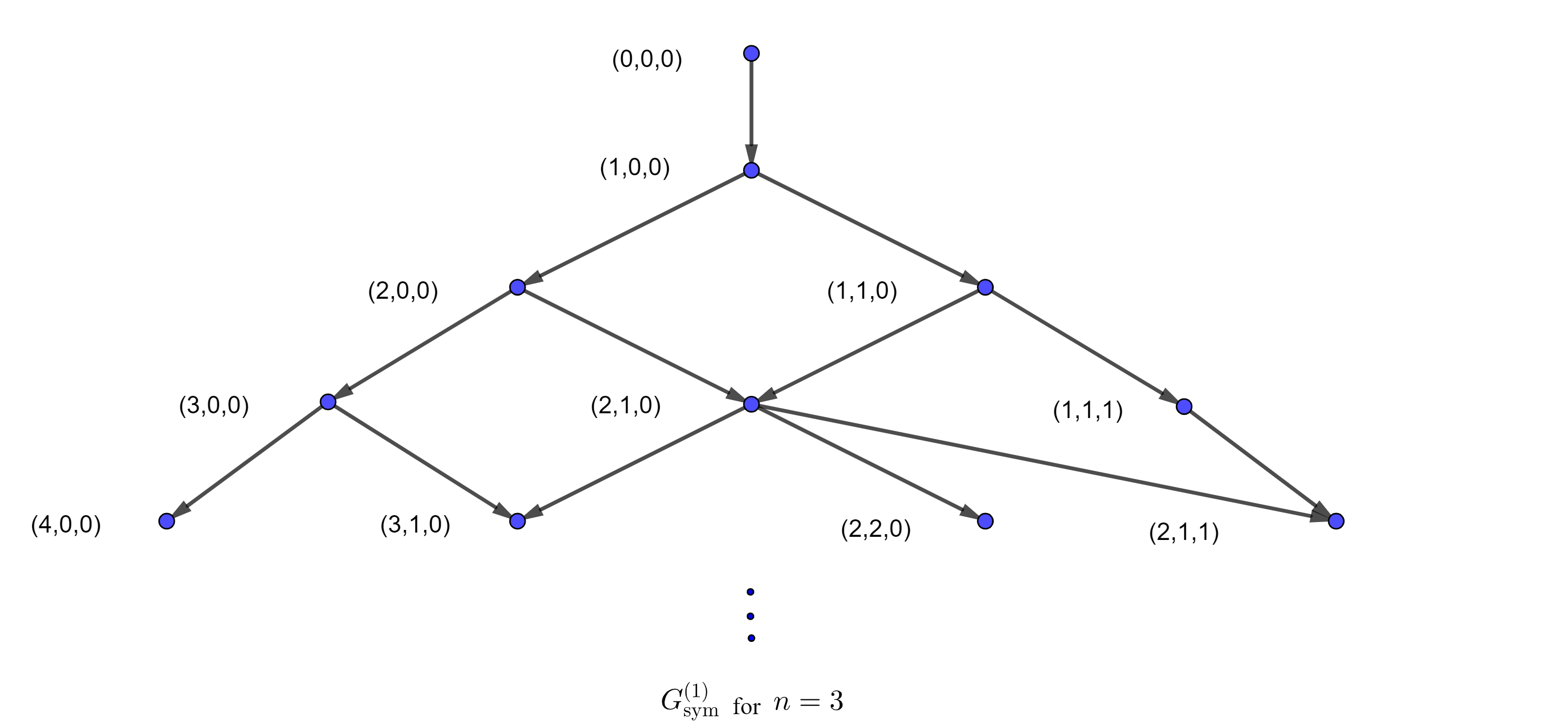}

\includegraphics[scale=.95]{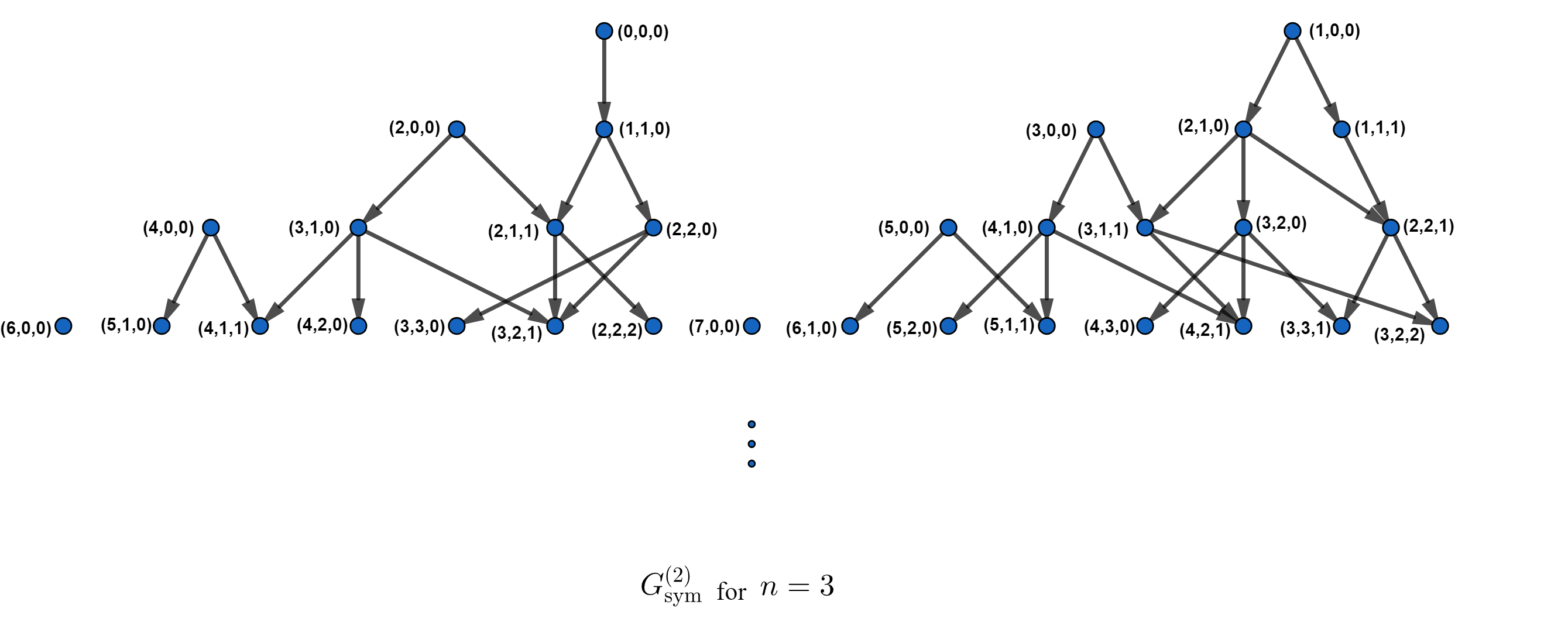}

\includegraphics[scale=.8]{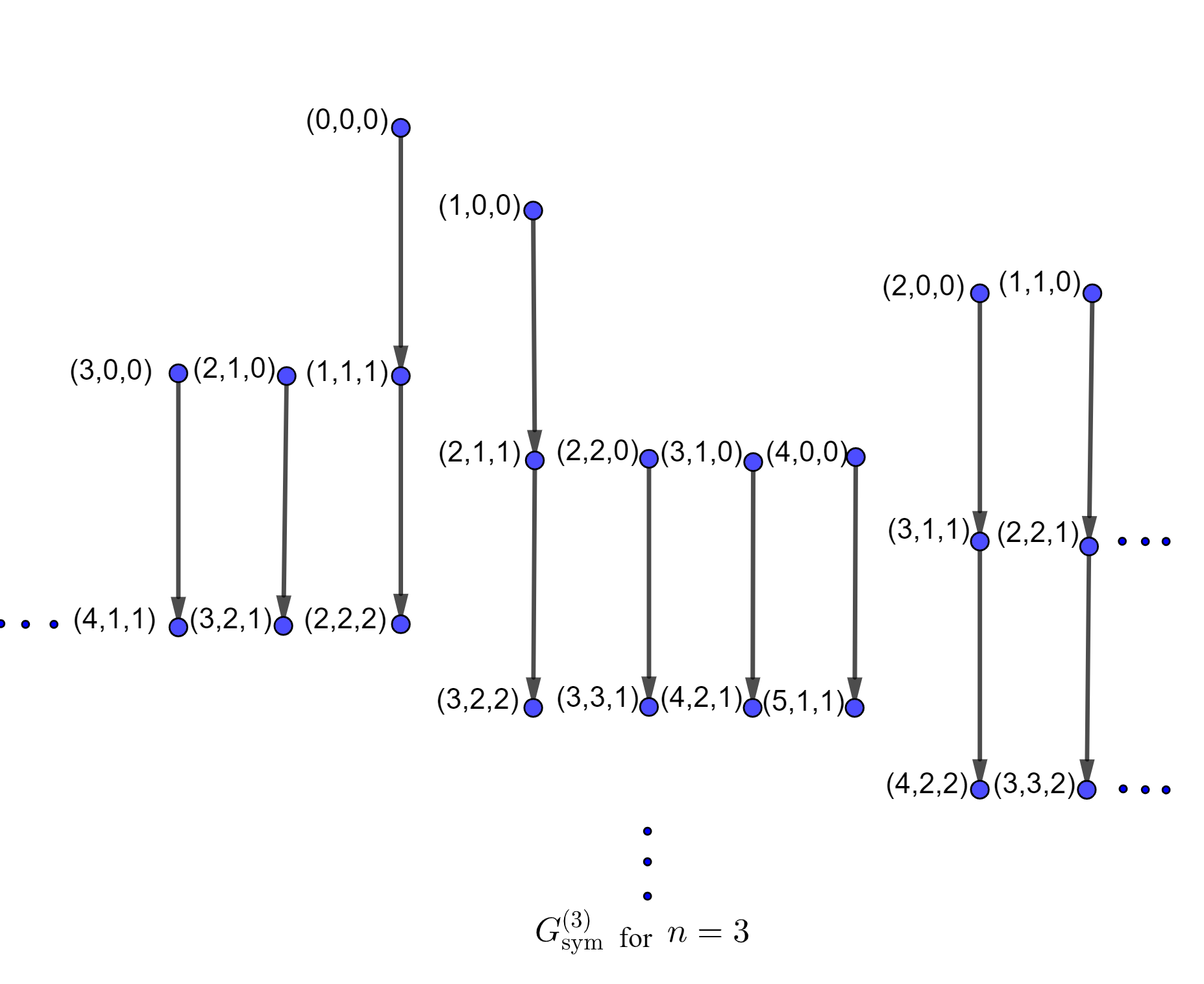}

\section{Example II}\label{section 5}

A \emph{pseudoreflection} on $\C^n$ is a linear homomorphism $\sigma: \C^n \rightarrow \C^n$ such that $\sigma$ has finite order in $GL(n,\mb C)$ and the rank of ${\rm id} - \sigma$ is 1. A group generated by pseudoreflections is called a pseudoreflection group.

Let $G$ be a finite pseudoreflection group. The group action of $G$ on $\C^n$ is defined by $(\sigma, \bl z)\mapsto \sigma\cdot\bl z = \sigma^{-1}\bl z,\, \bl z\in\C^n,\,\sigma\in G.$ Let $G$ act on the set of functions on  $\C^n$ by $\sigma (f)(\bl z)=f({\sigma}^{-1}\cdot \bl z).$ A function is said to be $G$-invariant if $\sigma(f)=f,$ for all $\sigma \in G.$ The ring of all complex polynomials in $n$ variables is denoted by $\mb C[z_1,\ldots,z_n]$. Moreover, the set of all $G$-invariant polynomials, denoted by $\mb C[z_1,\ldots,z_n]^G$, forms a ring. Chevalley, Shephard and Todd characterize finite pseudoreflection groups in the following theorem.
\begin{thm*}\cite[Theorem 3, p.112]{MR1890629}\label{A}
The ring $\C[z_1,\ldots,z_n]^G$ consisting of all $G$-invariant polynomials is equal to $\C[\theta_1,\ldots,\theta_n]$ and $\theta_i$ are algebraically independent
homogeneous polynomials if and only if $G$ is a finite pseudoreflection group.
\end{thm*}
The map ${\bl\theta}: \C^n \rightarrow \C^n$, defined by
\bea\label{theta}
{\bl\theta}(\bl z) = \big(\theta_1(\bl z),\ldots,\theta_n(\bl z)\big),\,\,\bl z\in\C^n
\eea is called a basic polynomial map associated to the group $G.$ The set of the degrees of the homogeneous polynomials $\theta_i$'s, $\{{\rm deg}(\theta_i)=\eta_i:i=1,\ldots,n\},$ is unique for the group $G.$ Let $\mb N_0$ denote the set of non-negative integers and $\bl \a=(\a_1,\ldots,\a_d)\in\mb N_0^n$ be a multi-index. For $\bl z=(z_1,\ldots,z_n)\in\mb C^n,$ denote $\bl z^{\bl\a}:=\prod_{j=1}^nz_j^{\a_j}.$  We consider $\theta_i(\bl z)=\sum_{\bl u \in N_i} a^{(i)}_{\bl u} \bl z^{\bl u}$, where $N_i \subseteq \{\bl u=(u_1,\dots,u_n) \in \mb N_0^n :\sum_{j=1}^n u_j = \eta_i\}$ and $a^{(i)}_{\bl u}$'s are positive numbers. Let $\ca(N_i)= c_i.$

Recall that a domain $U \subset \mb C^n$ (containing origin $\bl 0$) is said to be a complete Reinhardt domain with center $\bl 0$ if for any $\bl z =(z_1,\ldots,z_n) \in U,$ the domain $U$ contains the closure of the polydisc $D(\bl 0; \bl r),$ where $\bl r = (r_1,\ldots,r_n)$ and $|z_j| = r_j$ for $j=1,\ldots,n$ \cite[p. 21]{L}. Let $D$ be a bounded complete Reinhardt domain in $\mb C^n$. The Bergman space on $D,$ denoted by $\mb A^2(D)$, is the subspace of holomorphic functions in $L^2(D)$ with respect to the Lebesgue measure on $D.$ The Bergman space  $\mb A^2(D)$ is a  Hilbert space with a reproducing kernel, called the Bergman kernel. 
  The set $\{\frac{\bl z^{\bl\a}}{\norm{\bl z^{\bl\a}}}\}_{\bl \a \in \mb N_0^n}$ forms an orthonormal basis for $\mb A^2(D).$ The multiplication operator $M_{\theta_i} : \mb A^2(D) \to \mb A^2(D)$ is bounded for each $i=1,\ldots,n.$
Note that \bea \nonumber M_{\theta_i}\bl z^{\bl\a}  &=& \sum_{\bl u \in N_i} a^{(i)}_{\bl u} \bl z^{\bl u + \bl \a}\\  M_{\theta_i}\frac{\bl z^{\bl\a}}{\norm{\bl z^{\bl\a}}} &=& \sum_{\bl u \in N_i} a^{(i)}_{\bl u} \frac{\norm{\bl z^{\bl\a + \bl u}}}{\norm{\bl z^{\bl\a}}} \frac{\bl z^{\bl\a + \bl u}}{\norm{\bl z^{\bl\a + \bl u}}}. \eea
For each $i=1,\ldots,n,$ we fix the notation $E_n^{(i)} = \{(\bl \a,\bl \a+\bl u) : \bl \a \in \mb N_0^n \text{~and~} \bl u \in N_i\}$. 
\begin{thm}
For every $i=1,\ldots,n,$ the graph $(\mathcal G_i,c'_i) = (\mb N_0^n, E_n^{(i)})$ is a generalized directed semi-tree, for some $c'_i \leq c_i$. 
\end{thm}
\begin{proof}
We fix $i\in \{1,\ldots,n\}$ and set ${\rm Chi}_i(\bl u) = \{\bl v :(\bl u, \bl v) \in E_n^{(i)}\}.$ Then for every $\bl u, \bl v \in \mb N_0^n,$ $\ca({\rm Chi}_i(\bl u) \cap {\rm Chi}_i(\bl v)) \leq \ca({\rm Chi}_i(\bl u)) \leq c_i.$ We suppose that ${\rm sup}_{\bl u, \bl v \in \mb N_0^n, \bl u \neq \bl v} \ca({\rm Chi}_i(\bl u) \cap {\rm Chi}_i(\bl v)) = c'_i.$ Evidently, the supremum is attained for some $\bl u,\bl v$ and $\ca({\rm Chi}_i(\bl u) \cap {\rm Chi}_i(\bl v)) \leq c'_i \leq c_i$ for every $\bl u, \bl v \in \mb N_0^n$. Similar arguments as Proposition \ref{semtres} prove that these graphs can have at most countable components and they have no circuit. Hence, the result follows.
\end{proof}
Let $\l^{(i)}_{(\bl \a, \bl \a + \bl u)} := a^{(i)}_{\bl u} \frac{\norm{\bl z^{\bl\a + \bl u}}}{\norm{\bl z^{\bl\a}}}.$ For each $i=1,\ldots,n,$ we define an weighted shift $\Lambda_i : \ell^2(\mb N_0^n) \to \ell^2(\mb N_0^n)$ on $(\mathcal G_i,c'_i)$ by \bea \Lambda_i \chi_{\bl \a} = \sum_{\bl u \in N_i} \l^{(i)}_{(\bl \a, \bl \a + \bl u)} \chi_{\bl \a +\bl u},\,\,\,\, \bl \a \in \mb N_0^n, \eea where $\chi_{\bl \a}$ is the characteristic function in $\ell^2(\mb N_0^n)$ at $\bl \a.$   The set $\{\chi_{\bl \a}\}_{\bl \a \in \mb N_0^n}$ forms an orthonormal basis for the Hilbert space $\ell^2(\mb N_0^n).$
\begin{prop}
For each $i=1,\ldots,n,$ the weighted shift $\Lambda_i$ on the generalized directed semi-tree $(\mathcal G_i,c'_i)$ is unitarily equivalent to $M_{\theta_i}$ on $\mb A^2(D).$ 
\end{prop}
\begin{proof}
The unitary operator $U:\ell^2(\mb N_0^n) \to \mb A^2(D),$ defined by $$U(\chi_{\bl \a}) = \frac{\bl z^{\bl\a}}{\norm{\bl z^{\bl\a}}},\,\, \bl \a \in \mb N_0^n,$$ intertwines the operator $\Lambda_i$ on $\ell^2(\mb N_0^n)$ and $M_{\theta_i}$ on $\mb A^2(D),$ for all $i=1,\ldots,n.$ 
\end{proof}
\subsection{On the Bergman space of bidisc}
We consider the bounded Reinhardt domain bidisc $\mb D^2.$ The domain $\mb D^2$ is closed under the action of the symmetric group $\mathfrak{S}_2$ on two symbols. The symmetrization map $\bl s := (s_1,s_2): \mb D^2 \to \bl s(\mb D^2)$ is a basic polynomial map associated to the group $\mathfrak{S}_2,$ where $$s_1(z_1,z_2) = z_1+z_2 \text{~and~} s_2(z_1,z_2) = z_1z_2, \text{~for~} (z_1,z_2) \in \mb D^2.$$ From the above discussion, we get that the multiplication operators $M_{s_i} : \mb A^2(\mb D^2) \to \mb A^2(\mb D^2),\,\, i=1,2,$ yield two generalized directed semi-trees $(\mathcal G_1,c'_1)$ and $(\mathcal G_2,c'_2)$ which are generated with respect to the orthonormal basis $\{a_{(n_1,n_2)} = \frac{z_1^{n_1}z_2^{n_2}}{\sqrt{(n_1+1)(n_2+1)}}: (n_1,n_2) \in \mb N_0^2\}$ for $\mb A^2(\mb D^2),$ that is,

\begin{flalign*}
M_{z_1+z_2} a_{(n_1,n_2)} &= \frac{\sqrt{(n_1+2)}}{\sqrt{(n_1+1)}} a_{(n_1+1,n_2)} + \frac{\sqrt{(n_2+2)}}{\sqrt{(n_2+1)}} a_{(n_1,n_2+1)},\\ M_{z_1z_2} a_{(n_1,n_2)} &=  \frac{\sqrt{(n_1+2)(n_2+2)}}{\sqrt{(n_1+1)(n_2+1)}} a_{(n_1+1,n_2+1)}.
\end{flalign*}
Moreover, it is clear from the above expression that $c'_1$ can be at most $1$ and $c'_2=0.$ Note that both $(n_1,n_2+1)$ and $(n_1+1,n_2)$ have $(n_1+1,n_2+1)$ as one of their children, so $\ca\big(\textnormal{Chi}((n_1,n_2+1)) \cap \textnormal{Chi}((n_1+1,n_2))\big)=1$ for $n_1,n_2 \in \mb N_0.$ Thus $c'_1 = 1.$ We provide pictorial descriptions for $(\mathcal{G}_1,1)$ and $(\mathcal{G}_2,0)$ below. 

\includegraphics[scale=0.9]{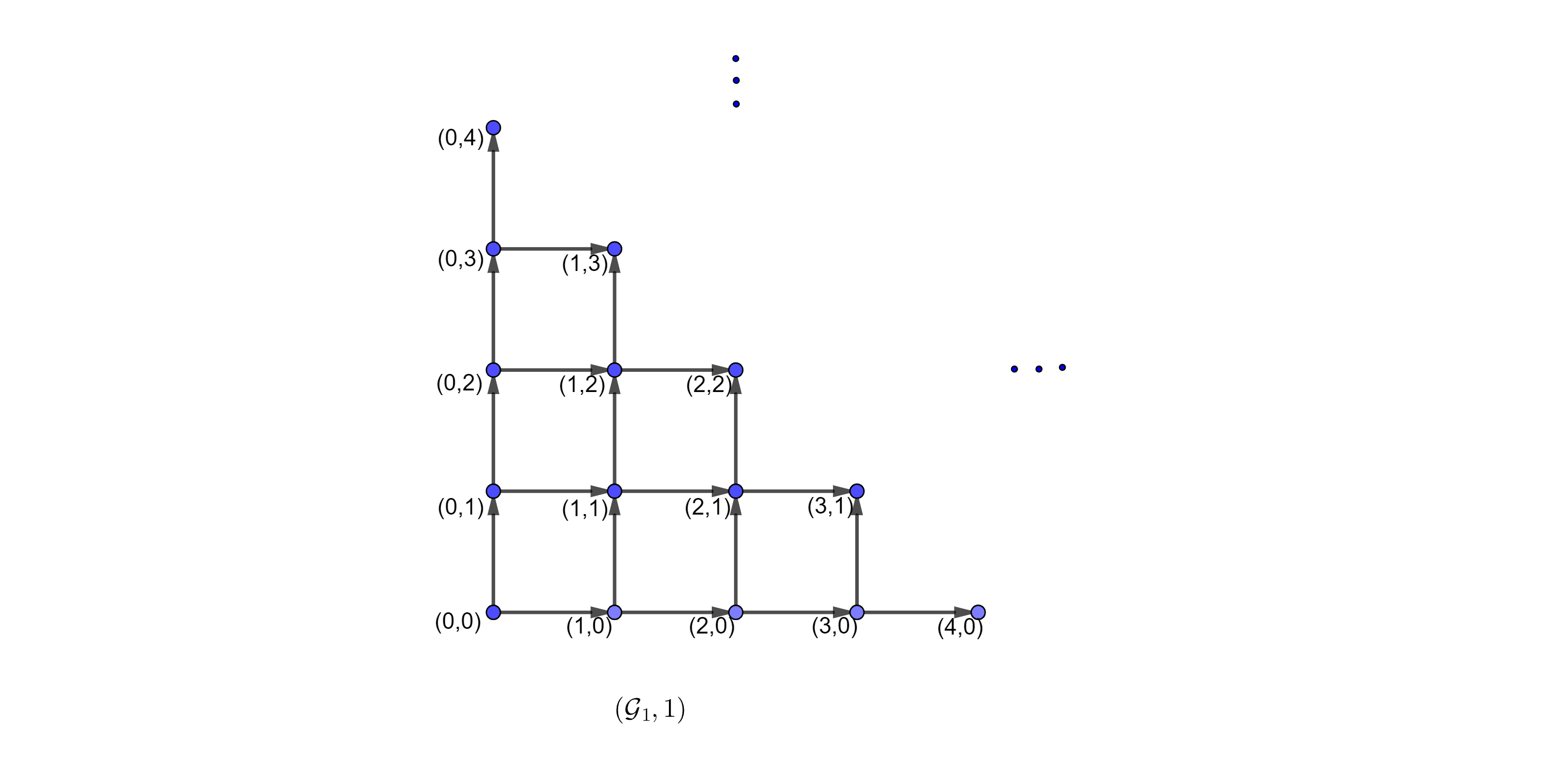}

\includegraphics[scale=1.2]{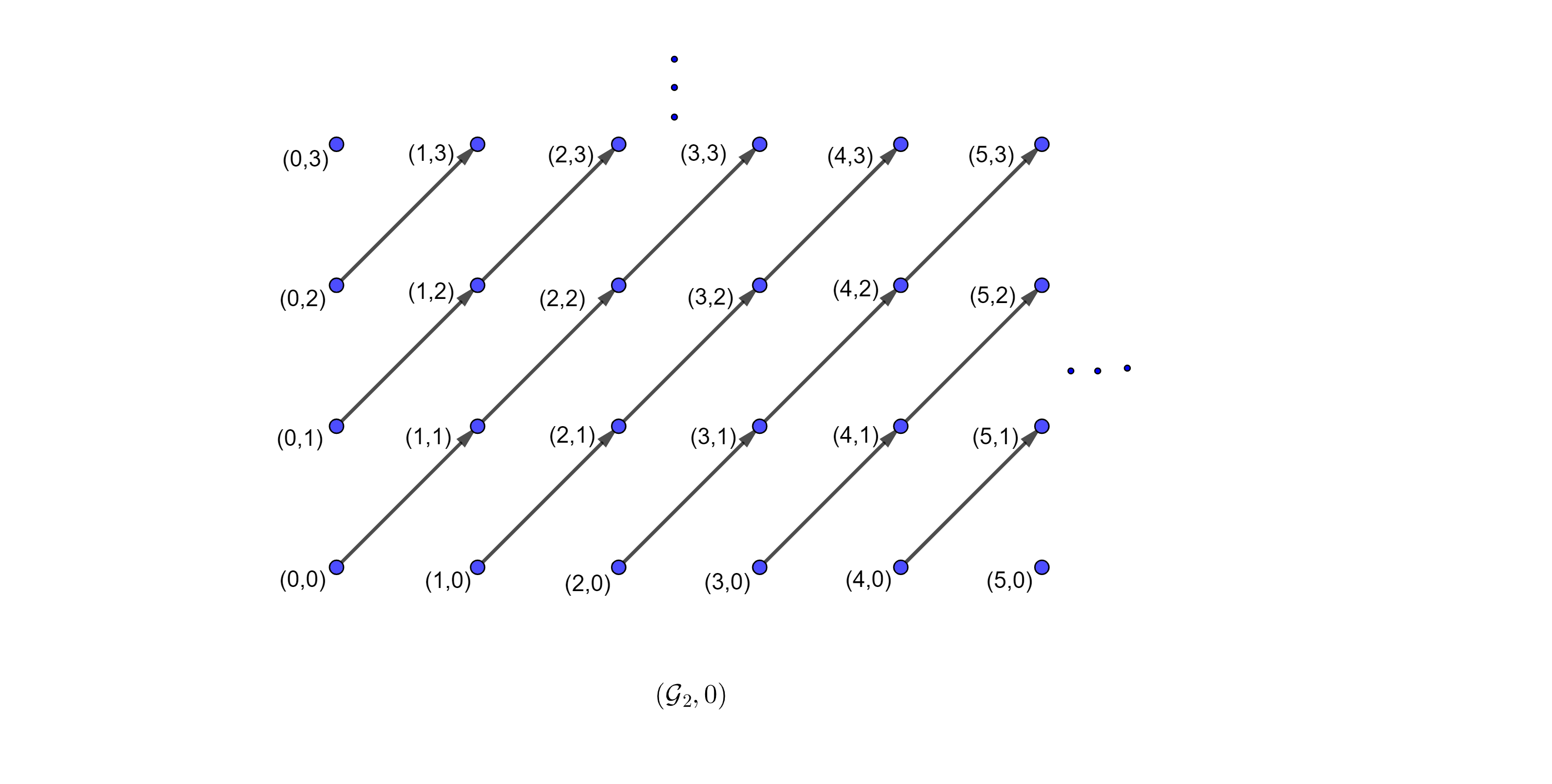}

\medskip \textit{Acknowledgement}.
The authors would like to express their sincere gratitude to Subrata Shyam Roy for several comments and suggestions in the preparation of this paper. The authors are grateful to the anonymous referee for many useful comments and suggestions.

\end{document}